\documentclass{article}

\title{Filling in the semantics of intuitionistic conditional logic}
\author{Brendan Dufty$^1$ \and Jim de Groot$^2$}
\date{\footnotesize{%
    $^1$\texttt{brendan.dufty@du.edu}, University of Denver, Denver, U.S.A.\\%
    $^2$\texttt{jim.degroot@unibe.ch}, University of Bern, Bern, Switzerland\\[2ex]%
}}

\usepackage{amsmath}
\usepackage{amssymb}
\usepackage{amsthm}
  \theoremstyle{definition}
  \swapnumbers
    \newtheorem{para}{}[section]
    \newtheorem{definition}[para]{Definition}
    \newtheorem{example}[para]{Example}
    \newtheorem{remark}[para]{Remark}
    \newtheorem{question}[para]{Question}
  \theoremstyle{plain}
    \newtheorem{lemma}[para]{Lemma}
    
    \newtheorem{corollary}[para]{Corollary}
    \newtheorem{theorem}[para]{Theorem}
    \newtheorem{proposition}[para]{Proposition}

\usepackage{txfonts}
\usepackage[dvipsnames]{xcolor}
\usepackage[colorlinks=true,
            linktocpage=true,
            citecolor=RubineRed,
            linkcolor=RoyalBlue]{hyperref}
\usepackage[numbers]{natbib}
\usepackage{doi}
\usepackage[a4paper,margin=1in]{geometry}
\usepackage{booktabs}
\usepackage{multicol}
\usepackage{proof}
\usepackage{tikz-cd}
  \usetikzlibrary{arrows,arrows.meta}
  \tikzcdset{arrow style=tikz, diagrams={>=stealth'}}

\newcommand{\lan}{\mathcal}
\newcommand{\mb}{\mathbb}
\newcommand{\mc}{\mathcal}
\newcommand{\mf}{\mathfrak}

\renewcommand{\phi}{\varphi}
\newcommand{\cto}{\boxright}
\newcommand{\dto}{\boxdotright}
\newcommand{\up}{\mathrm{up}}
\newcommand{\Prop}{\mathrm{Prop}}
\newcommand{\power}{\mathcal{P}}
\renewcommand{\iff}{\quad\text{iff}\quad}
\renewcommand{\log}[1]{\mathbf{#1}}
\newcommand{\even}{\mathrm{even}}
\newcommand{\Ax}{\mathrm{Ax}}
\newcommand{\ickAx}{\mathrm{iCK\text{-}Ax}}
\newcommand{\IntAx}{\mathrm{Int\text{-}Ax}}
\newcommand{\ikAx}{\mathrm{iK\text{-}Ax}}
\newcommand{\can}{\mathbb{C}_\Ax}
\DeclareMathOperator{\Cn}{Cn}
\newcommand{\says}{\mathrel{\mathtt{says}}}

\newcommand{\Boxi}{\Box_i}
\newcommand{\Boxm}{\Box_m}

\newcommand{\Diamondm}{\Diamond_m}

\newcommand{\ckappa}{\kappa_{\mathrm{caut}}}

\newcommand{\deriv}[1]{\vdash_{\scriptstyle {#1}}}
\newcommand{\myitem}[1]{%
    \renewcommand{\labelenumi}{(\theenumi) }
    \renewcommand{\theenumi}{#1}
    \item}

\newcommand{\ick}{\mathbf{iCK}}
\newcommand{\icc}{\mathbf{iCC}}

\newcommand{\mix}{\mathsf{mix}}

\newcommand{\axKc}{\text{$\mathsf{k_c}$}}
\newcommand{\axNc}{\text{$\mathsf{n_c}$}}

\newcommand{\axMP}{\text{$\mathsf{mp}$}}
\newcommand{\axCS}{\text{$\mathsf{cs}$}}

\newcommand{\axRefl}{\text{\hyperlink{ax:refl}{$\mathsf{id}$}}}
\newcommand{\axCt}{\text{$\mathsf{ct}$}}
\newcommand{\axCm}{\text{$\mathsf{cm}$}}
\newcommand{\axCa}{\text{$\mathsf{ca}$}}

\newcommand{\axF}{\text{\hyperlink{ax:4}{$\mathsf{\oldstylenums{4}_c}$}}}

\newcommand{\axCem}{\text{\hyperlink{ax:cem}{$\mathsf{cem_1}$}}}
\newcommand{\axCemm}{\text{\hyperlink{ax:cem2}{$\mathsf{cem_2}$}}}
\newcommand{\axCemmm}{\text{\hyperlink{ax:cem3}{$\mathsf{cem_3}$}}}
\newcommand{\axEcm}{\text{\hyperlink{ax:ecm}{$\mathsf{ecm_1}$}}}
\newcommand{\axEcmm}{\text{\hyperlink{ax:ecm2}{$\mathsf{ecm_2}$}}}

\newcommand{\axUnit}{\text{\hyperlink{ax:unit}{$\mathsf{unit}$}}}
\newcommand{\axC}{\text{\hyperlink{ax:bt}{$\mathsf{self}$}}}
\newcommand{\axBind}{\text{\hyperlink{ax:bt}{$\mathsf{bind}$}}}
\newcommand{\axCf}{\text{\hyperlink{ax:c4}{$\mathsf{idem}$}}}
\newcommand{\axTrust}{\text{\hyperlink{ax:c4}{$\mathsf{trust}$}}}

\begin{document}

\maketitle

\begin{abstract}
\noindent
We prove completeness results for a wide variety of intuitionistic conditional logics. We do so by first using a canonical model construction obtain completeness with respect to descriptive conditional frames, and then introducing the fill-in method to transfer this to classes of conditional frames without extra structure. The fill-in method closes the gap between descriptive conditional frames, which do not have a canonical underlying frame, and conditional frames.
\end{abstract}

\section{Introduction}

\paragraph{Conditionals and conditional logic}
  Conditional logics have historically arisen from the gap between classical
  (material) implication, and the use of ``implies'' or ``if \ldots then''
  in natural language. The truth-functional interpretation of material implication declares
  an implication true if its antecedent is false or its consequence is true, which
   is often misaligned with our intuition.
  For example, one may deny that
  \medskip
  \begin{center}
    ``if God exists then the prayers of evil men will be answered.''
  \end{center}
  \medskip
  The truth-functional interpretation of implication then leads
  us to ``conclude that God exists, and (as a bonus) we may conclude that the
  prayers of evil men will not be answered''~\cite[Section~II]{Ste70}.
  Other such paradoxes are given by e.g.~MacColl~\cite{Mac08}
  and Lewis~\cite{Lew12}.
  
  The flexibility of natural language allows us to
  use implications to capture a variety of concepts such as causality,
  possibility and future action.
  Each of these relies on a certain \emph{relatedness} to capture the idea
  that when $\phi$ implies $\psi$, then $\phi$ should somehow be related to $\psi$.
  We often describe this by saying that $\psi$ is \emph{conditional} on $\phi$.
  Besides dropping truth-functionality, such conditionals often
  violate other axioms true for material implication.
  Consider for example the statements
  \medskip
  \begin{center}
    ``If I quit my job, then I won't be able to buy a house''
  \end{center}
  \medskip
  and
  \medskip
  \begin{center}
    ``If I win the lottery, then I will quit my job.''
  \end{center}
  \medskip
  Then transitivity would lead us to believe that
  \medskip
  \begin{center}
    ``If I win the lottery, then I won't be able to buy a house,''
  \end{center}
  \medskip
  which would clearly be a dubious conclusion.
  So our notion of a conditional need not be transitive.
  More counterexamples for other axioms can be found in
  e.g.~\cite{Ada65,Ste70,Lew73,Ber22}.
  
  The logical formalisation of non-material implications resulted in a wide
  variety of logics.
  One of the earliest (modern) appearances of a non-truth-functional implication
  is Lewis' \emph{strict implication}~\cite{Lew13,Lew14,Lew18}.%
  \footnote{Discussions about the meaning of implication date back to the Roman Stoics, see e.g.~\cite{Mos24} for references.}
  Two decades later, both Reichenbach and De Finetti adopted the perspective that
  the truth of a conditional should be indeterminate when its antecedent is false,
  leading to three-valued semantics~\cite{Rei35,Fin36,EgrRosSpr21a}.
  The idea that conditionals express a probabilistic dependence between
  their antecedents and their consequents was investigated by Ramsey~\cite{Ram31}
  and others (see e.g.~\cite{Fin36,Ada65,Edg95}).

  The modern landscape of \emph{conditional logic} is often said to have
  started with work on counterfactuals by Stalnaker~\cite{Stal68},
  Lewis~\cite{Lew73} and Adams~\cite{Ada75}.
  Stalnaker captured the idea of relatedness between the antecedent and
  consequent of a conditional using possible world semantics
  with a \emph{selection function} that, given a world $w$ and the
  antecedent $\phi$, selects a world $s(w, \phi)$ which is related or relevant
  to $w$. The conditional $\phi \cto \psi$ is then defined to be true at
  $w$ if $\psi$ is true at $s(w, \phi)$.
  This idea was adapted to allow $s(w, \phi)$ to be a set of worlds,
  rather than a single world~\cite{Che75},
  the idea still being that $s(w, \phi)$ ``selects'' a set of worlds
  related to $\phi$ at $w$.
  This notion of a selection function can also be viewed as having
  a proposition-indexed set of relations~\cite{Seg89},
  with $\phi \cto \psi$ being true at $w$ if all $R_{\phi}$-successors of
  $w$ satisfy $\psi$. Its philosophical interpretation is studied in e.g.~\cite{Unt13}.

  From then on, the study of conditional logic has been an active area
  of research~\cite{Gab85,Vel85,Bou90,CroLam92,WanUnt19,Wei19b},
  with semantic results ranging from the
  finite model property~\cite{Seg89} to
  correspondence theorems~\cite{UntSch14}
  to bisimulations~\cite{BalCin18},
  as well as algebraic~\cite{Nut75,Fra76,Nut79,Nut88,GuzSqu90} and proof-theoretic 
  investigations~\cite{Gir19,GirNegSba19,GriNegOli21,EgrRosSpr21b}.

\paragraph{Conditionals in an intuitionistic setting}
  While the majority of research on conditional logics is carried out
  over a classical propositional basis, some authors have adopted an
  intuitionistic setting.
  Philosophically, some argue that intuitionistic logic is the correct model
  of reasoning (such as~Dummet \cite{Dum84}),
  while from a computer science perspective this brings one to a
  constructive setting.
  For example, in~\cite{LitVis18,LitVis24} intuitionistic logic is
  extended with a Lewis-style strict implication,
  resulting in a logic corresponding e.g.~to type inhabi\-tation of
  Hughes arrows and to an
  intuitionistic epistemic account of entailment.
  Furthermore, various access control logics have an intuitionistic conditional
  flavour (see Example~\ref{exm:access-control}).

  In another recent development, Weiss introduced an intuitionistic counterpart of basic
  conditional logic~\cite{Wei19}.
  This was obtained by
  extending intuitionistic logic with a binary conditional operator that
  satisfies the same axioms and rules as the conditional operator in~\cite{Che75}.
  Semantics for the logic are given by means of intuitionistic
  Kripke frames with a set of relations indexed by sets of worlds.
  Furthermore, Weiss gives G\"{o}del-style and
  Glivenko-style embeddings between intuitionistic conditional logic
  and its classical counterpart, as well as an embedding of 
  intuitionistic modal logic with a box modality into intuitionistic
  conditional logic.
  Slightly modified system of basic intuitionistic conditional logic
  that additionally contain a (diamond-like) conditional possibility operator
  were proposed in~\cite{Olk24} and~\cite{DalGir26}.

  Based on Weiss' groundwork, Ciardelli and Liu provided completeness results for certain
  extensions of basic intuitionistic conditional logic~\cite{Liu18,CiaLiu20}.
  Semantically they deviate slightly from Weiss' work:
  in~\cite{Liu18} the conditional relations are indexed by formulas, rather
  than sets of worlds, while the frames in~\cite{CiaLiu20} carry a collection of
  admissible sets in which formulas must be interpreted (we call such
  structures general conditional frames, see Definition~\ref{def:gf}).
  This move to general frames simplifies the canonical model construction,
  because it prevents the need to define all conditional relations.
  In the current paper, we return to an intuitionistic counterpart
  of the semantics used by Chellas and Segerberg.

\paragraph{Our contributions}
  The goal of this paper is to prove \emph{Kripke completeness} for
  intuitionistic conditional logics. That is, we prove completeness 
  with respect to relational semantics
  that do not rely on topology or a set of admissible subsets.
  To this end, we first recall a canonical model construction from~\cite{CiaLiu20}
  that gives rise to strong completeness of any axiomatic extension of
  intuitionistic conditional logic with respect to classes of
  general conditional frames.
  In such frames, the selection function is only defined on admissible upsets.
  Therefore, simply deleting the collection of admissible sets leaves
  ``gaps'' in the selection function.
  
  Thus, the challenge is to extend a selection function that is only defined
  on admissible upsets to all upsets without losing validity of the axioms. 
  We investigate two ways of doing so.
  The first way extends a selection function by assigning value $\emptyset$ whenever it is undefined.
  Axioms that preserve validity under this move are called
  \emph{$\kappa_{\emptyset}$-persistent}, in analogy with \emph{d-persistence}
  in normal modal logic (see e.g.~\cite[Section~2]{WolZak99}
  and \cite[Definition~5.84]{BlaRijVen01}).
  We exploit a connection with intuitionistic normal modal logic to
  establish $\kappa_{\emptyset}$-persistence for a wide variety of axioms,
  resulting in Kripke completeness for extensions of intuitionistic
  conditional logic with any collection of such axioms.

  But not all axioms are $\kappa_{\emptyset}$-persistent.
  In particular, this is the case for the extension of intuitionistic
  conditional logic with \emph{cautious} analogues of
  identity, transitivity and monotonicity that are widely used in the
  conditional logic literature.
  In order to capture them, we propose a second
  way to fill in the gaps: the \emph{cautious fill-in}. This is an example of a more intricate
  fill-in definition, and gives rise to Kripke completeness for several
  more extensions of intuitionistic conditional logic including the
  cautious axioms.
  After having seen these two examples of fill-ins, we briefly
  discuss several other possible fill-ins and the types of axioms
  they accommodate.

\section{Intuitionistic normal modal logic}

  We briefly recall basic notions of intuitionistic logic and
  intuitionistic normal modal logic.

\subsection{Intuitionistic logic}

  Let $\lan{L}$ be the language given by the grammar
  \begin{equation*}
    \phi ::= p
      \mid \bot
      \mid \phi \wedge \phi
      \mid \phi \vee \phi
      \mid \phi \to \phi,
  \end{equation*}
  where $p$ ranges over some arbitrary but fixed set $\Prop$ of proposition
  letters. A \emph{consecution} is an expression of the form
  $\Gamma \vdash \phi$, where $\Gamma \cup \{ \phi \} \subseteq \lan{L}$.
  Let $\IntAx$ consist of the following axioms,
  \begin{align*}
    &(p \wedge q) \to p
    &&(p \wedge q) \to q
    &&(p \to q) \to ((p \to r) \to (p \to (q \wedge r))) \\
    &p \to (p \vee q)
    &&q \to (p \vee q)
    &&(p \to r) \to ((q \to r) \to ((p \vee q) \to r)) \\
    &\bot \to p
    &&p \to (q \to p)
    &&(p \to (q \to r)) \to ((p \to q) \to (p \to r))
  \end{align*}
  and write $\mc{I}(\IntAx)$ for the set of substitution instances of axioms in $\IntAx$.
  Define the axiomatic system $\log{Int}$ by
  \begin{multicols}{3}%
  \begin{enumerate}
    \setlength{\itemindent}{1.1em}
    \renewcommand{\labelenumi}{(\theenumi) }
    \renewcommand{\theenumi}{$\mathsf{Ax}$}
    \item \label{rule:Ax-int}
          $\dfrac{\phi\in\mathcal{I}(\IntAx)}{\Gamma \vdash \phi}$
    \renewcommand{\theenumi}{$\mathsf{MP}$}
    \item \label{rule:MP-int}
          $\dfrac{\Gamma \vdash \phi \qquad \Gamma \vdash \phi \to \psi}{\Gamma \vdash \psi}$
    \renewcommand{\theenumi}{$\mathsf{El}$}
    \item \label{rule:El-int}
          $\dfrac{\phi \in \Gamma}{\Gamma \vdash \phi}$
  \end{enumerate}
  \end{multicols}
  \noindent
  We say that $\Gamma \vdash \phi$ is \emph{provable in} $\log{Int}$,
  and write $\Gamma \vdash_{\mathrm{Int}} \phi$,
  if there exists a tree of consecutions built using the rules above
  with $\Gamma \vdash \phi$ as root and adequate applications of \eqref{rule:El-int} and \eqref{rule:Ax-int} as leaves.
  
  An \emph{(intuitionistic) Kripke frame} is a nonempty poset,
  i.e.~a nonempty set with a reflexive, antisymmetric and transitive binary relation.
  If $(X, \leq)$ is a poset set and $a \subseteq X$, then we define
  ${\uparrow}a := \{ x \in X \mid y \leq x \text{ for some } y \in a \}$.
  We say that $a$ is \emph{upward closed} or an \emph{upset} if $a = {\uparrow}a$,
  and define $\up(X, \leq)$ as the collection of upsets of $(X, \leq)$.
  Similarly, we define ${\downarrow}a$ and downsets.
  For $x \in X$, we abbreviate ${\uparrow}x = {\uparrow} \{ x \}$.

\begin{definition}\label{def:int-interpretation}
  A \emph{Kripke model} is a tuple $\mf{M} = (X, \leq, V)$ consisting of a
  Kripke frame $(X, \leq)$ and a valuation $V : \Prop \to \up(X, \leq)$.
  The interpretation of $\phi \in \lan{L}$ at a world $x$ in $\mf{M}$
  is defined recursively by:
  \begin{align*}
    \mf{M}, x \models p &\iff x \in V(p) \\
    \mf{M}, x \models \bot &\phantom{\iff} \text{never} \\
    \mf{M}, x \models \phi \wedge \psi &\iff \mf{M}, x \models \phi \text{ and } \mf{M}, x \models \psi \\
    \mf{M}, x \models \phi \vee \psi &\iff \mf{M}, x \models \phi \text{ or } \mf{M}, x \models \psi \\
    \mf{M}, x \models \phi \to \psi &\iff \forall y \in X (\text{if } x \leq y \text{ and }\mf{M}, y \models \phi \text{ then } \mf{M}, y \models \psi)
  \end{align*}
  The \emph{truth set} of $\phi$
  is denoted by $V(\phi) = \{ x \in X \mid \mf{M}, x \models \phi \}$.
\end{definition}

  The Kripke frames and models we use in this paper are also called
  \emph{intuitionistic} Kripke frames and models. Since we deal exclusively
  with intuitionistic frames and models, we drop the prefix ``intuitionistic.''
  Sometimes it is useful to restrict denotations of formulas to a
  predefined collection of upsets:
  
\begin{definition}
  A \emph{general Kripke frame} is a tuple $\mb{X} = (X, \leq, A)$
  consisting of a Kripke frame $(X, \leq)$ and a set $A \subseteq \up(X, \leq)$
  that contains $\emptyset$ and $X$, and is closed under binary intersections
  and unions and under the operation
  \begin{equation*}
    {\Rightarrow}
      : \up(X, \leq) \times \up(X, \leq) \to \up(X, \leq)
      : (a, b) \mapsto \{ x \in X \mid {\uparrow}x \cap a \subseteq b \}.
  \end{equation*}
  (In other words, if $a, b \in A$ then $a \Rightarrow b \in A$.)
  The members of $A$ are called \emph{admissible} upsets of $\mb{X}$.
  
  An \emph{admissible valuation} is a function $V : \Prop \to A$ that assigns
  to each proposition letter an admissible upset.
  A \emph{general Kripke model} is a general Kripke frame together with an
  admissible valuation. Truth of formulas in a general Kripke model
  is defined as in Definition~\ref{def:int-interpretation}.
\end{definition}

  Restricting the class of general Kripke frames further, to so-called
  descriptive Kripke frames, gives rise to a duality
  for the category of Heyting algebras, the algebraic semantics of intuitionistic
  logic~\cite{Esa74,Esa19}. For the purpose of this paper, it suffices to know
  the definition of a descriptive Kripke frame.
  If $(X, \leq, A)$ is a general Kripke frame, then we write
  $-A = \{ X \setminus a \mid a \in A \}$ for the complements of the members of $A$.

\begin{definition}
  A \emph{descriptive Kripke frame} is a general Kripke frame $\mf{X} = (X, \leq, A)$
  that additionally satisfies
  \begin{enumerate}\itemindent=1.1em
    \myitem{D$_{\mathrm{cpt}}$} \label{it:d-cpt}
          if $\mc{B} \subseteq A \cup -A$ has the finite intersection property,
          then $(\bigcap \mc{B}) \neq \emptyset$;
    \myitem{D$_{\to}$} \label{it:d-to}
          ${\uparrow}x = \bigcap \{ a \in A \mid x \in a \}$ for every $x \in X$.
  \end{enumerate}
  \emph{Descriptive Kripke models} are general Kripke models based on a
  descriptive Kripke frame.
\end{definition}

\subsection{Adding a normal modality}

  The language $\lan{L}_{\Box}$ is obtained by extending the grammar of
  $\lan{L}$ with a unary operator $\Box$, i.e.~it is given by
  \begin{equation*}
    \phi ::= p
        \mid \bot 
        \mid \phi \wedge \phi
        \mid \phi \vee \phi
        \mid \phi \to \phi
        \mid \Box\phi,
  \end{equation*}
  where $p$ ranges over some set $\Prop$ of proposition letters.

\begin{definition}
  For $\Ax \subseteq \lan{L}_{\Box}$, define the axiomatic system
  $\log{iK} \oplus \Ax$ by
  \begin{equation*}
    \mathsf{(Ax)} \;\;
          \dfrac{\phi\in\mathcal{I}(\ikAx) \cup \mc{I}(\Ax)}{\Gamma \vdash \phi} \qquad\quad
    \mathsf{(MP)} \;\;
          \dfrac{\Gamma \vdash \phi \qquad \Gamma \vdash \phi \to \psi}{\Gamma \vdash \psi} \qquad\quad
    \mathsf{(El)} \;\;
           \dfrac{\phi \in \Gamma}{\Gamma \vdash \phi} \qquad\quad
    \mathsf{(Nec)} \;\;
      \dfrac{\emptyset \vdash \phi}{\Gamma \vdash \Box\phi}
  \end{equation*}
  where $\mc{I}(\ikAx)$ consists of all substitution instances of the formulas in $\IntAx$ and the axiom
  $\Box(p \wedge q) \leftrightarrow \Box p \wedge \Box q$,
  and $\mc{I}(\Ax)$ contains all substitution instances of the formulas in $\Ax$.
  We abbreviate $\mathbf{iK} := \mathbf{iK} \oplus \emptyset$.
\end{definition}

  The language $\lan{L}_{\Box}$ can be interpreted in relational structures
  that we will call \emph{modal frames}.

\begin{definition}
  A \emph{modal (Kripke) frame} is a tuple $(X, \leq, R)$ consisting of a
  poset $(X, \leq)$ and a binary relation $R$ on $X$ such that
  $({\leq} \circ R \circ {\leq}) = R$.
  A \emph{modal model} is a tuple $\mf{M} = (X, \leq, R, V)$ consisting of
  a modal frame $(X, \leq, R)$ and a valuation $V$,
  and $\lan{L}_{\Box}$-formulas can be interpreted in modal
  models via the clauses from Definition~\ref{def:int-interpretation} and
  \begin{equation*}
    \mf{M}, x \models \Box\phi \iff \forall y \in X (x R y \text{ implies } \mf{M}, y \models \phi).
  \end{equation*}
  We write $R[x] := \{ y \in X \mid xRy \}$,
  so $x \models \Box\phi$ iff $R[x] \subseteq V(\phi)$.
\end{definition}

\begin{definition}
  A \emph{general modal frame} is a tuple $(X, \leq, R, A)$ such that
  $(X, \leq, R)$ is a modal frame, $(X, \leq, A)$ is a general Kripke frame,
  and for every $a \in A$ we have $\boxdot a := \{ x \in X \mid R[x] \subseteq a \} \in A$.
  It is called a \emph{descriptive modal frame} if additionally $(X, \leq, A)$ is a
  descriptive Kripke frame and
  \begin{enumerate}\itemindent=1.1em
    \renewcommand{\labelenumi}{(\theenumi) }
    \renewcommand{\theenumi}{D$_{\Box}$}
    \item $R[x] = \bigcap \{ a \in A \mid R[x] \subseteq a \}$ for all $x \in X$.
  \end{enumerate}
\end{definition}

  General and descriptive modal models are defined as expected.
  The following notion of \emph{d-persistence} is used
  to derive completeness results for extensions of $\mathbf{iK}$~\cite{WolZak98}.

\begin{definition}
  If $\mb{X} = (X, \leq, R, A)$ is a descriptive modal frame, then 
  $\kappa\mb{X} = (X, \leq, R)$ denotes the \emph{underlying modal frame}.
  A formula $\phi \in \lan{L}_{\Box}$ is called \emph{d-persistent}
  if $\mb{X} \models \phi$ implies $\kappa\mb{X} \models \phi$ for ever
  descriptive modal frame~$\mb{X}$.
\end{definition}

  One way of proving d-persistence is via the G\"{o}del-McKinsey-Tarski (GMT) translation
  from $\lan{L}_{\Box}$ into a classical bimodal logic.
  Let $\lan{L}_{\Boxi\Boxm}$ be the extension of classical propositional logic
  with modal operators $\Boxi$ and $\Boxm$. Then the GMT translation given by
  \begin{align*}
    t(p) &= \Boxi p &\text{for $p \in \Prop \cup \{ \bot \}$} \\
    t(\phi \star \psi) &= \Boxi(t(\phi) \star t(\psi))
      &\text{for $\star \in \{ \wedge, \vee \to \}$} \\
    t(\Box\phi) &= \Boxi\Boxm t(\phi)
  \end{align*}
  It then follows from~\cite{WolZak97,WolZak98} that an $\lan{L}_{\Box}$-formula
  $\phi$ is d-persistent whenever its translation into $\lan{L}_{\Boxi\Boxm}$
  is Sahlqvist:
  
\begin{theorem}\label{thm:d-pers-Sahl}
  Suppose $\phi \in \lan{L}_{\Box}$ is a formula whose
  G\"odel translation $t(\phi)$ is Sahlqvist. Then $\phi$ is d-persistent.
\end{theorem}
\begin{proof}[Proof sketch]
  The axiom $\mix$ is defined by $\Boxi\Boxm\Boxi q \leftrightarrow \Boxm q$.
  By~\cite[Theorem~9]{WolZak98} the logic
  $(\log{S4} \otimes \log{K}) \oplus t(\phi) \oplus \mix$ is a
  so-called \emph{modal companion} of $\log{iK} \oplus \phi$,
  and by \cite[Theorem~12]{WolZak98} the latter is d-persistent if and only
  if its modal companion is (in the classical sense of~\cite[Definition~5.84]{BlaRijVen01}).
  Since every Sahlqvist formula is d-persistent, the assumption
  together with the fact that $\mix$ and the S4-axioms are d-persistent
  implies that $\phi$ is d-persistent.
\end{proof}

  For more complicated formulas one can use (an implementation of) the
  SQEMA algorithm~\cite{ConGorVak06} to verify whether or not $t(\phi)$
  is d-persistent or not.

\section{Intuitionistic conditional logic}

\subsection{The axioms and rules of intuitionistic conditional logic}

  Let $\mathcal{L}_{\cto}$ be the language generated by the grammar
  \begin{equation*}
    \phi ::= p \mid \bot \mid \phi \wedge \phi \mid \phi \vee \phi \mid \phi \to \phi \mid \phi \cto \phi,
  \end{equation*}
  where $p \in \Prop$.
  We abbreviate $\neg\phi := \phi \to \bot$, $\top = \neg\bot$
  and $(\phi \leftrightarrow \psi) := (\phi \to \psi) \wedge (\psi \to \phi)$.
  Hencefort, \emph{consecutions} are expressions $\Gamma \vdash \phi$ such that
  $\Gamma \cup \{ \phi \} \subseteq \lan{L}_{\cto}$.

  The minimal conditional modality studied in the literature is a binary operator
  that satisfies the congruence rule for both arguments, i.e.~we can
  replace equals with equals, and preserves finite meets in its second argument.
  Intuitionistic conditional logic~\cite{Wei19} is obtained by extending
  intuitionistic logic with such a modality.
  In the literature, $\cto$ is sometimes denoted by
  $\Rightarrow$~\cite{Che75,Che80}, $\sqsupset$~\cite{Seg89},
  $>$~\cite{Nut88,CiaLiu20} and $\rightsquigarrow$~\cite{BalCin18}.

\begin{definition}
  Let $\ickAx$ be an axiomatisation of intuitionistic logic
  together with the axioms
  \begin{multicols}{2}
  \begin{enumerate}\itemindent=1.1em
	\item[(\axKc)] $(p \cto (q \land r)) \leftrightarrow ((p \cto q) \land (p \cto r))$
	\item[(\axNc)] $(p \cto \top)$
  \end{enumerate}
  \end{multicols}
  \noindent
  For a set of formulas $\Lambda$, we write $\mathcal{I}(\Lambda)$ for the collection of substitution instances of formulas in $\Lambda$.
  For a set $\Ax \subseteq \mathbf{L}$,
  define the axiomatic system $\log{CK \oplus \Ax}$ by:
  \medskip
  \begin{multicols}{3}%
  \begin{enumerate}
    \setlength{\itemindent}{1.1em}
    \renewcommand{\labelenumi}{(\theenumi) }
    \renewcommand{\theenumi}{$\mathsf{Ax}$}
    \item \label{rule:Ax}
          $\dfrac{\phi\in\mathcal{I}(\ickAx)\cup \mathcal I (\Ax)}{\Gamma \vdash \phi}$
    \renewcommand{\theenumi}{$\mathsf{MP}$}
    \item \label{rule:MP}
          $\dfrac{\Gamma \vdash \phi \qquad \Gamma \vdash \phi \to \psi}{\Gamma \vdash \psi}$
    \renewcommand{\theenumi}{$\mathsf{El}$}
    \item \label{rule:El}
          $\dfrac{\phi \in \Gamma}{\Gamma \vdash \phi}$
  \end{enumerate}
  \end{multicols}
  \begin{multicols}{2}
  \begin{enumerate}
    \setlength{\itemindent}{1.1em}
    \renewcommand{\labelenumi}{(\theenumi) }
    \renewcommand{\theenumi}{$\mathsf{Ax}$}
    \renewcommand{\theenumi}{$\mathsf{Rcea}$}
    \item \label{rule:Rcea}
          $\dfrac{\emptyset \vdash \phi \leftrightarrow \psi}
                 {\Gamma \vdash (\phi \cto \chi) \leftrightarrow (\psi \cto \chi)}$
    \renewcommand{\theenumi}{$\mathsf{Rcec}$}
    \item \label{rule:Rcec}
          $\dfrac{\emptyset \vdash \phi \leftrightarrow \psi}
                 {\Gamma \vdash (\chi \cto \phi) \leftrightarrow (\chi \cto \psi)}$
  \end{enumerate}
  \end{multicols}
  
\medskip\noindent
  Provability of $\Gamma \vdash \phi$ in $\ick \oplus \Ax$ is defined as expected, and denoted by $\Gamma \deriv{\Ax} \phi$.

  We write
  $\ick \oplus \phi_1 \oplus \dots \oplus \phi_n$ for $\ick \oplus \{ \phi_1, \ldots, \phi_n \}$,
  and we write $\Gamma \vdash_{\Ax} \Delta$ if there exists 
  $\psi_1, \ldots, \psi_m \in \Delta$ such that
  $\Gamma \vdash_{\Ax} \psi_1 \vee \cdots \vee \psi_m$.
\end{definition}

\begin{proposition}
  The following rules, where $\sigma$ is a uniform substitution, are admissible
  in $\ick \oplus \Ax$ for any set of axioms $\Ax$:
  \begin{equation*}
    \infer{\Gamma, \Gamma' \vdash \phi}
          {\Gamma \vdash \phi}
    \qquad
    \infer{\Gamma \vdash \phi}
          {\{\Gamma\vdash\delta \; \mid \; \delta\in\Delta\} \quad \Delta\vdash\varphi}
    \qquad
	\infer{\Gamma^\sigma\vdash\varphi^\sigma}
	      {\Gamma\vdash\varphi}
	\qquad
    \infer={\Gamma\vdash\phi\rightarrow\psi}
          {\Gamma,\phi\vdash\psi}
  \end{equation*}
\end{proposition}
\begin{proof}
  By induction on the height of a derivation for the premiss(es).
\end{proof}

  The three topmost rules show that $\ick \oplus \Ax$ is a monotone,
  compositional and structural relation, respectively.
  Furthermore, using the fact that proof trees are finite
  we can show that $\Gamma \deriv{\Ax} \phi$ if and only if
  there is a finite $\Gamma' \subseteq \Gamma$ such that $\Gamma' \deriv{\Ax} \phi$.
  Therefore $\ick \oplus \Ax$ is a finitary logic~\cite[Definition~1.4.1]{Kra99}.

\begin{example}\label{exm:access-control}
  In computer systems, access control is concerned with the decision of
  accepting or denying a request from a user or program to perform an operation
  on e.g.~a file.
  To reason about such systems, they are often modelled by logics with a
  conditional operator denoted $\cto$ (also denoted by by ``$\says$''),
  where $\phi \cto \psi$ is read as
  ``$\phi$ asserts or supports $\psi$''~\cite{AbaBurLamPlo93,LiGroFei03,BoeGabGenTor09}.
  Sometimes the antecedent is restricted to a fixed set of proposition
  letters denoting agents of the system.
  To prevent the derivability of unwanted conclusions,
  it has been proposed that such logics should be built on an
  intuitionistic basis~\cite{GarPfe06,Aba08,GarAba08,GenGioGliPoz14}.
  Common additional axioms include:
  \begin{multicols}{2}
  \begin{enumerate}\itemindent=1.1em \itemsep=0em
    \item[(\axUnit)] $q \to (p \cto q)$
    \item[(\axTrust)] $(\bot \cto q) \to q$
    \item[(\axBind)] $(q \to (p \cto r)) \to ((p \cto q) \to (p \cto r))$
    \item[(\axCf)] $(p \cto (p \cto q)) \to (p \cto q)$
    \item[(\axC)] $p \cto ((p \cto q) \to q)$
    \item[]
  \end{enumerate}
  \end{multicols}
  \noindent
  The right two, for instance, can be read as follows:
  $\axC$ states that $p$ says whatever it proclaims true must be true,
  and $\axCf$ states that $p$ says what it claims to say.
  There are many variations of access control logic, modifying or extending
  the ones discussed above, see e.g.~\cite{DeT02,Aba07,GenGioGliPoz14}.
\end{example}

\begin{example}\label{exm:icc}
  As discussed in the introduction, axioms such as monotonicity and
  transitivity are often too strong for conditional reasoning.
  However, there is a set of weaker forms of these axioms that is widely
  accepted, namely:
  \begin{enumerate}\itemindent=1.1em \itemsep=0em
    \item[(\axCt)] $(p \cto q) \wedge ((p \wedge q) \cto r) \to (p \cto r)$
           \hfill (\emph{cautious transitivity})
    \item[(\axCm)] $(p \cto q) \wedge (p \cto r) \to ((p \wedge q) \cto r)$
		   \hfill (\emph{cautious monotonicity})
  \end{enumerate}
  Both axioms are mentioned is Chellas' \emph{basic conditional logic}~\cite{Che75},
  and they are included in a wide variety of conditional logics,
  see e.g.~\cite{Ada75,Gab85,KraLehMag90,Ben03,PfeKle10,Bad21},
  as well as topic-sensitive modals~\cite{Ber18b,Bad21,Ber22,OzgCot25}.
  A third cautious axiom sometimes found in the literature~\cite{HawOzg23}
  is:
  \begin{enumerate}\itemindent=1.1em
    \item[(\axCa)] $(p \cto q) \to (p \cto (p \wedge q))$
          \hfill(\emph{cautious agglomeration})
  \end{enumerate}
  Since over $\ick$ this is equivalent to \axRefl,
  we favour the simpler formulation of $\axRefl$.

  In the context of imagination, the cautious axioms are often assumed
  to be true~\cite{Bad21,Ber22,Ferguson_2024,OzgunBerto_2021}.
  As a minimal intuitionistic counterpart, we define a the logic
  of the \textbf{i}ntuitionistic \textbf{c}autious \textbf{c}onditional by
  \begin{equation*}
    \log{iCC} := \ick \oplus {\axRefl} \oplus \axCt \oplus \axCm.
  \end{equation*}
\end{example}

\subsection{Frames}

One mechanism for interpreting the conditional is via a
  \emph{selection function}. In the classical setting, a selection function
  for a nonempty set $X$ is a function $s : X \times \power(X) \to \power(X)$,
  where $\power(X)$ denotes the powerset of $X$.
  The intuition is that given a world $x$ and
  a proposition $a \subseteq X$ such a function selects the set of worlds relevant to $a$ at $x$.
  A formula of the form $\phi \cto \psi$ is true at $x$ if
  every world selected by $s$ on inputs $x$ and $V(\phi)$
  satisfies $\psi$.
  A selection function can also be encoded as a ternary relation
  $T \subseteq X \times \power(X) \times X$ (used in e.g.~\cite{UntSch14})
  or as a collection of relations $\mc{R} = \{ R_a \mid a \in \mc{P}(X) \}$
  (as in~\cite{Seg89}). The latter perspective aligns with the intuition
  that the conditional implication $\phi \cto \psi$ behaves as a box modality
  $[\phi]\psi$ indexed by its first argument.
  
  The simplest approach to turn this intuitionistic is to define an frames as triples
  $(X, \leq, s)$ consisting of an intuitionistic Kripke frame $(X, \leq)$
  and a classical selection function $s : X \times \power(X) \to \power(X)$~\cite{Wei19}.
  However, since denotations of formulas are given by upsets,
  a selection function of the form $s : X \times \up(X, \leq) \to \up(X, \leq)$
  suffices. 

\begin{definition}\label{def:cond-frm}
  Let $(X, \leq)$ be a Kripke frame. An \emph{intuitionistic selection function}
  on $(X, \leq)$ is a function
  \begin{equation*}
    s : X \times \up(X, \leq) \to \up(X, \leq)
  \end{equation*}
  such that $x \leq y$ implies $s(y, u) \subseteq s(x, u)$ for all
  $x, y \in X$ and $u \in \up(X, \leq)$.
  A \emph{conditional frame} is a a tuple $(X, \leq, s)$ consisting of
  a Kripke frame $(X, \leq)$ and an intuitionstic selection function $s$.
\end{definition}

  Models are defined by extending a conditional frame with a valuation.

\begin{definition}
  A \emph{conditional model} is a pair $\mf{M} = (\mf{F}, V)$ such that
  $\mf{F} = (X, \leq, s)$ is a conditional frame and $V : \Prop \to \up(X, \leq)$ is
  a valuation of the proposition letters. Satisfaction of a formula
  $\phi \in \lan{L}_{\cto}$ at a world $x$ is defined recursively via the
  clauses from Definition~\ref{def:int-interpretation} and
  \begin{equation*}
    \mf{M}, x \models \phi \cto \psi
      \iff s(x, V(\phi)) \subseteq V(\psi),
  \end{equation*}
  where $V(\phi) := \{ x \in X \mid \mf{M}, x \models \phi \}$ is the
  \emph{truth set} of $\phi$.

  Let $\Gamma \cup \{ \phi \} \subseteq \lan{L}_{\cto}$
  and let $\mf{M}$ be a conditional model.
  We write $\mf{M}, x \models \Gamma$ if $x$ satisfies all $\psi \in \Gamma$,
  and we say that $\mf{M}$ \emph{validates} $\Gamma \vdash \phi$ if
  $\mf{M}, x \models \Gamma$ implies $\mf{M}, x \models \phi$ for all worlds
  $x$ in $\mf{M}$.
  A conditional frame $\mf{X}$ \emph{validates} $\Gamma \vdash \phi$ if every model
  of the form $(\mf{X}, V)$ validates the consecution, and it validates a
  formula $\phi$ if it validates the consecution $\emptyset \vdash \phi$.
  If $\Ax \subseteq\lan{L}_{\cto}$ is a set of axioms, then
  we write $\Gamma \models_{\Ax} \phi$,
  and say that \emph{$\Gamma$ semantically entails $\phi$ on the class of
  conditional frames for $\ick \oplus \Ax$}, if every conditional frame that
  validates all formulas in $\Ax$ also validates the consecution $\Gamma \vdash \phi$.
\end{definition}

  In~\cite{CiaLiu20}, the frames used to interpret intuitionistic conditional
  logic are equipped with a collection of \emph{admissible}
  upsets. These correspond to what we call \emph{general frames}:

\begin{definition}\label{def:gf}
  Let $(X, \leq, A)$ be a general Kripke frame.
  A \emph{general selection function} on $(X, \leq, A)$ is a function
  \begin{equation*}
    s : X \times A \to \up(X, \leq)
  \end{equation*}
  such that $x \leq y$ implies $s(y, a) \subseteq s(x, a)$
  for all $x, y \in X$ and $a \in A$.
  A \emph{general conditional frame} is a tuple
  $(X, \leq, s, A)$ such that $(X, \leq, A)$ is a general Kripke frame,
  $s$ is a general selection function, and $A$ is closed under:
  \begin{equation*}
    \dto
      : \up(X, \leq) \times \up(X, \leq) \to \up(X, \leq)
      : (a, b) \mapsto \{ x \in X \mid s(x, a) \subseteq b \}
  \end{equation*}

  A \emph{general conditional model} is a general conditional frame $\mb{X} = (X, \leq, s, A)$
  together with an admissible valuation, i.e.~a valuation that assigns
  to each proposition letter $p$ and upset $V(p) \in A$.
  A general conditional frame $\mb{X}$ \emph{validates} a consecution
  $\Gamma \vdash \phi$ if every general conditional model of the form
  $(\mb{X}, V)$ validates $\Gamma \vdash \phi$.

  If $\Ax \subseteq\lan{L}_{\cto}$ is a set of axioms, then
  we write $\Gamma \models_{\Ax}^g \phi$
  if every general conditional frame that validates all formulas in $\Ax$
  also validates the consecution $\Gamma \vdash \phi$.
\end{definition}
  
  Conditional frames (and models) correspond bijectively with 
  general conditional frames (models) in which $A$ is the collection
  of all upsets. Such general conditional frames are called \emph{full}.
  Using this identification, we can transfer propositions that apply
  to general conditional frames to the class of conditional frames.
  For instance, if we specify the following proposition to full
  general conditional frames it states the usual results that
  denotations of formulas in conditional models are always given by upsets.
  
\begin{proposition}
  Let $\mb{X} = (X, \leq, s, A)$ be a general conditional frame and
  $\mb{M} = (\mb{X}, V)$ a general conditional model.
  Then $V(\phi)$ is an upset of $(X, \leq)$ for every $\phi \in \lan{L}_{\cto}$.
\end{proposition}
\begin{proof}
  This follows from a routine induction on the structure of $\phi$,
  similar to~\cite[Proposition~1]{CiaLiu20}.
\end{proof}

  In order to transfer results from intuitionistic normal modal logic
  to intuitionistic conditional logic, we use the following notion of
  a descriptive conditional frame.
  While beyond the scope of this paper, we note that this can be
  used to give a duality for the algebraic semantics of $\ick$,
  which appeared in disguise in~\cite[Section~12.4]{Gro22-phd}.

\begin{definition}
  A \emph{descriptive conditional frame} is a general conditional frame
  $(X, \leq, s, A)$ such that $(X, \leq, A)$ is a descriptive Kripke frame
  and:
  \begin{enumerate}\itemindent=1.1em
    \myitem{D$_{\cto}$} \label{it:d-cto}
          $s(x, a) = \bigcap \{ b \in A \mid s(x, a) \subseteq b \}$
          for every $x \in X$ and $a \in A$.
  \end{enumerate}
  We write $\Gamma \models_{\Ax}^d \phi$
  if every descriptive conditional frame that validates all formulas in $\Ax$
  also validates $\Gamma \vdash \phi$.
\end{definition}
  
  The next example of a descriptive conditional frame will be used in
  Section~\ref{sec:fill-in} to prove non-persistency results.

\begin{example}\label{exm:descr}
  Let $N$ be the set of natural numbers, $N_{\infty} = N \cup \{ \infty \}$, and
  let
  \begin{equation*}
    A = \{ a \subseteq N \mid a \text{ is finite} \} 
      \cup \{ a \cup \{ \infty \} \mid (X \setminus a) \subseteq N \text{ is infinite} \}.
  \end{equation*}
  Then $(N_{\infty}, =, A)$ is a descriptive Kripke frame.
  Define
  \begin{equation*}
    s : X \times A \to \up(X, \leq)
      : (x, a) \mapsto \{ y \in a \mid y = \infty \text{ or } y \text{ is even} \}.
  \end{equation*}
  Since the definition of $s(x, a)$ does not depend on $x$,
  for any $a, b \in A$ we have $a \dto b = \emptyset$ or $a \dto b = N_{\infty}$,
  both of which are admissible. Moreover, it is easy to verify that
  $s(x, a) = \bigcap \{ b \in A \mid s(x, a) \subseteq b \}$ for any
  $x \in N_{\infty}$ and $a \in A$.
  So $(N_{\infty}, =, s, A)$ is a descriptive conditional frame.  
\end{example}

  For future reference we prove a correspondence result for the logic
  $\icc$ from Example~\ref{exm:icc}.

\begin{lemma}\label{lem:sick-fr-corr}
  A full or descriptive conditional frame $\mb{X} = (X, \leq, s, A)$
  validates $\icc$ if and only if it satisfies:
  \begin{enumerate}
  \setlength{\itemindent}{3em}
    \renewcommand{\labelenumi}{\textup{(}\theenumi\textup{)} }
    \renewcommand{\theenumi}{$\mathsf{id}$-corr}
    \item \label{ax:refl-corr}
          $s(x, a) \subseteq a$, for all $x \in X$ and $a \in A$;
    \renewcommand{\theenumi}{$\mathsf{icc}$-corr}
    \item \label{ax:sick-corr}
          $s(x, a) \subseteq b \subseteq a$ implies $s(x, a) = s(x, b)$, for all $x \in X$ and $a, b \in A$.
  \end{enumerate}
\end{lemma}
\begin{proof}
  It is straightforward to see that \axRefl\ corresponds on \eqref{ax:refl-corr},
  so we focus on the other axioms.
  Suppose $\mb{X}$ validates $\icc$.
  Let $x \in X$, and $a, b \in A$ be such that $s(x, a) \subseteq b \subseteq a$.
  Let $V(p) = a$ and $V(q) = b$ and let $V(r) = c \in A$ for some arbitrary $c$.
  Then $x \models p \cto q$,
  so combining $\axCm$ and $\axCt$ shows that
  $x \models ((p \wedge q) \cto r) \leftrightarrow (p \cto r)$.
  Since $V(p) \cap V(q) = a \cap b = b$ and $V(r) = c$, we find
  \begin{equation*}
    s(x, b) \subseteq c \iff s(x, a) \subseteq c,
  \end{equation*}
  so using the fact that $\mb{X}$ is full or descriptive gives
  \begin{equation*}
    s(x, b)
      = \bigcap \{ c \in A \mid s(x, b) \subseteq c \}
      = \bigcap \{ c \in A \mid s(x, a) \subseteq c \}
      = s(x, a).
  \end{equation*}
  Therefore $\mb{X}$ satisfies \eqref{ax:sick-corr}.
  
  For the converse, suppose that $\mb{X}$ satisfies~\eqref{ax:refl-corr}
  and~\eqref{ax:sick-corr}.
  We show that $\mb{X}$ validates \axCt; the case for \axCm\ is similar.
  Let $V$ be any valuation for $\mb{X}$ and $x$ a world.
  Suppose $x \models (p \cto q) \wedge ((p \wedge q) \cto r)$.
  Then $s(x, V(p)) \subseteq V(q)$ and $s(x, V(p) \cap V(q)) \subseteq V(r)$.
  By~\eqref{ax:refl-corr}, we have
  $s(x, V(p)) \subseteq V(p) \cap V(q) \subseteq V(p)$,
  hence \eqref{ax:sick-corr} implies that $s(x, V(p)) = s(x, V(p) \cap V(q))$.
  Therefore $s(x, V(p)) \subseteq s(x, V(p)\cap V(q)) \subseteq V(r)$, so that
  $x \models p \cto r$ and $\mb{X} \models \axCt$.
\end{proof}

  As usual, we say that a logic is \emph{sound} with respect to a class of
  frames if every derivable consecution is valid on every frame in the class.

\begin{theorem}\label{thm:soundness}
  For any $\Ax \subseteq \lan{L}_{\cto}$,
  The logic $\ick \oplus \Ax$ is sound with respect to the classes
  of conditional frames, general conditional frames and descriptive
  conditional frames that validate $\Ax$.
\end{theorem}
\begin{proof}
  Since conditional frames can be viewed as a special type of general conditional
  frame, and descriptive conditional frames are general conditional frames,
  it suffices to prove that
  $\Gamma \vdash_{\Ax} \phi$ implies $\Gamma \models_{\Ax}^g \phi$
  for any $\Gamma \cup \{ \phi \} \subseteq \lan{L}_{\cto}$.
  This follows from a routine induction on the height of the derivation
  of $\Gamma \vdash \phi$.
\end{proof}

\subsection{Canonical model construction}\label{sec:canonical-model}

  We recall the canonical model construction from~\cite{CiaLiu20},
  reformulated to the selection function point of view.
  This proves completeness of (extensions of) $\ick$ with respect to
  classes of general conditional frames.
  Throughout this subsection, we fix some set $\Ax \subseteq \lan{L}_{\cto}$
  of axioms.

\begin{definition}
  An \emph{$\Ax$-prime theory} is a collection $\Gamma \subseteq \lan{L}_{\cto}$
  of formulas such that:
  \begin{enumerate}
    \item $\Gamma$ is \emph{deductively closed},
          i.e.~$\Gamma \vdash_{\Ax} \phi$ implies $\phi \in \Gamma$
          for all $\phi \in \lan{L}_{\cto}$;
    \item $\Gamma$ is \emph{consistent}, i.e.~$\bot \notin \Gamma$;
    \item $\Gamma$ is \emph{prime}, i.e.~$\phi \vee \psi \in \Gamma$
          implies $\phi \in \Gamma$ or $\psi \in \Gamma$
          for all $\phi \vee \psi \in \lan{L}_{\cto}$.
  \end{enumerate} 
\end{definition}

  We write $X_{\Ax}$ for the set of $\Ax$-prime theories.
  Furthermore, for any $\phi \in \lan{L}_{\cto}$ we define
  $\tilde{\phi} := \{ \Gamma \in X_{\Ax} \mid \phi \in \Gamma \}$.
  The Lindenbaum lemma can be proved as usual.

\begin{lemma}[Lindenbaum lemma]\label{lem:lindenbaum}
  If $\Delta \cup \Sigma \subseteq \mathcal{L}_\cto$
  and $\Delta \not\vdash_{\Ax} \Sigma$, then there exists an
  $\Ax$-prime theory $\Gamma$ such that $\Delta \subseteq \Gamma$
  and $\Gamma \not\vdash_{\Ax} \Sigma$.
\end{lemma}

  Thus armed we can define the canonical frame and model for $\ick \oplus \Ax$ as follows:

\begin{definition}
  The \textit{canonical frame} for $\ick \oplus \Ax$ is the general frame
  $\can = (X_{\Ax}, \subseteq, s_{\Ax}, A_{\Ax})$ where
  $X_{\Ax}$ is the set of $\Ax$-prime theories,
  $A_{\Ax} = \{ \tilde{\phi} \mid \phi \in \lan{L}_{\cto} \}$ and
  \begin{equation*}
    s_{\Ax}(\Gamma, \tilde{\phi}) = \bigcap \{ \tilde{\psi} \mid \phi \cto \psi \in \Gamma \}.
  \end{equation*}
  The \emph{canonical valuation} $V_{\Ax} : \Prop \to A_{\Ax}$ is given by
  $V_{\Ax}(p) = \tilde{p}$, and the \emph{canonical model} is
  $\mb{M}_{\Ax} := (\can, V_{\Ax})$.
\end{definition}

\begin{remark}
  In~\cite{CiaLiu20}, the modal part of the canonical frame is defined using
  the set $\Cn_{\phi}(\Gamma) := \{ \psi \in \lan{L}_{\cto} \mid \phi \cto \psi \in \Gamma \}$.
  Translated to our setting, the selection function is defined by
  $s_{\Ax}(\Gamma, \tilde{\phi}) = \{ \Delta \in X_{\Ax} \mid \Cn_{\phi}(\Gamma) \subseteq \Delta \}$.
  This is equivalent to our definition because
  $\Cn_{\phi}(\Gamma) \subseteq \Delta$ iff $\{ \psi \in \lan{L}_{\cto} \mid \phi \cto \psi \in \Gamma \} \subseteq \Delta$ iff $\Delta \in \tilde{\psi}$ for every $\psi \in \lan{L}_{\cto}$ such that $\phi \cto \psi \in \Gamma$.
\end{remark}

  Using a truth lemma as usual we then obtain~\cite[Lemma~3 and Proposition~6]{CiaLiu20}:

\begin{theorem}
  The logic $\ick \oplus \Ax$ is strongly complete with respect to the class of
  general conditional frames that validate $\Ax$.
\end{theorem}

  We extend this result in two ways: first we show that one can restrict to
  the class of \emph{descriptive} conditional frames.
  Second, we prove strong completeness of $\ick$ with respect to (non-general)
  conditional frames. We refer to this as \emph{(strong) Kripke completeness},
  because it does not rely on any additional structure on the frames.

\begin{lemma}
  The canonical frame $\can$ is a descriptive conditional frame.
\end{lemma}
\begin{proof}
  It follows from~\cite[Proposition~5]{CiaLiu20} that $\can$ is a general
  conditional frame, so we only have to show that it satisfies
  \eqref{it:d-cpt}, \eqref{it:d-to} and \eqref{it:d-cto}.
  The latter is true by definition.
  Since the prime theories are ordered by inclusion, for any $\Gamma \in X_{\Ax}$
  we have ${\uparrow}\Gamma \subseteq \bigcap \{ \tilde{\phi} \mid \phi \in \Gamma \}$.
  This inclusion is in fact an equality, because $\Gamma \not\subseteq \Delta$
  implies the existence of some $\phi \in \Gamma$ such that $\phi \notin \Delta$,
  witnessing $\Delta \notin \bigcap \{ \tilde{\phi} \mid \phi \in \Gamma \}$.
  This proves~\eqref{it:d-to}.
  
  Finally, for compactness, suppose $\mc{B} \subseteq A_{\Ax} \cup -A_{\Ax}$ has
  the finite intersection property.
  We use the Lindenbaum lemma to construct an $\Ax$-prime theory contained in
  $\bigcap \mc{B}$. To this end, let
  \begin{equation*}
    \Delta := \{ \phi \in \lan{L}_{\cto} \mid \tilde{\phi} \in \mc{B} \}
    \quad\text{and}\quad
    \Sigma := \{ \psi \in \lan{L}_{\cto} \mid (X_{\Ax} \setminus \tilde{\psi}) \in \mc{B} \}.
  \end{equation*}
  Suppose towards a contradiction that $\Delta \vdash_{\Ax} \Sigma$.
  Then there exist formulas $\phi_1, \ldots, \phi_n \in \Delta$
  and $\psi_1, \ldots, \psi_m \in \Sigma$ such that
  $\phi_1, \ldots, \phi_n \vdash_{\Ax} \psi_1, \ldots, \psi_m$.
  But this implies
  \begin{equation*}
    \tilde{\phi}_1 \cap \cdots \cap \tilde{\phi}_n \cap (X_{\Ax} \setminus \tilde{\psi}_1) \cap \cdots \cap (X_{\Ax} \setminus \tilde{\psi}_m) = \emptyset,
  \end{equation*}
  contradicting the assumption that $\mc{B}$ has the finite intersection property.
  So we have $\Delta \not\vdash_{\Ax} \Sigma$, and we can use
  Lemma~\ref{lem:lindenbaum} to find a prime theory $\Gamma$
  that contains $\Delta$ and is disjoint from $\Sigma$.
  This implies $\Gamma \in \bigcap\mc{B}$, as desired.
\end{proof}

  Since the canonical frame is a descriptive conditional frame,
  we find:
  
\begin{theorem}{\label{thm:descriptive-complete}}
  For any $\Ax \subseteq \lan{L}_{\cto}$, the logic $\ick \oplus \Ax$
  is strongly complete with respect
  to the class of descriptive conditional frames that validate $\Ax$.
\end{theorem}

  Finally, we prove Kripke completeness of $\ick$ with
  respect to the class of all conditional frames.

\begin{theorem}\label{thm:ick-Kripke-complete}
  The logic $\ick$ is strongly complete with respect to the class of conditional frames.
\end{theorem}
\begin{proof}
  Let $\Ax \cup \Gamma \cup \{ \phi \} \subseteq \lan{L}_{\cto}$ and suppose
  $\Gamma \not\vdash_{\Ax} \phi$. Then the canonical model $\mb{M}_{\Ax}$
  does not validate the consecution
  $\Gamma \vdash \phi$, so there exists a world $x \in X_{\Ax}$ such that
  $\mb{M}_{\Ax}, x \models \psi$ for all $\psi \in \Gamma$, but 
  $\mb{M}_{\Ax}, x \not\models \phi$.
  
  We need to find a conditional frame that invalidates $\Gamma \vdash \phi$.
  To achieve this, we define a (non-general) selection function $\hat{s}$
  on $(X_{\Ax}, \subseteq)$ by
  \begin{equation*}
    \hat{s}
      : X_{\Ax} \times \up(X_{\Ax}, \subseteq) \to \up(X_{\Ax}, \subseteq)
      : (x, u) \mapsto
        \begin{cases}
          s_{\Ax}(x, u) &\text{if } u \in A_{\Ax} \\
          \emptyset &\text{otherwise}
        \end{cases}
  \end{equation*}
  and define $\mf{F} := (X_{\Ax}, \subseteq, \hat{s})$.
  If $y, z \in X_{\Ax}$ are prime theories such that $y \subseteq z$,
  then either $\hat{s}(z, u) = \emptyset \subseteq \emptyset = \hat{s}(y, u)$
  or $\hat{s}(z, u) = s_{\Ax}(z, u) \subseteq s_{\Ax}(y, u) = \hat{s}(y, u)$,
  so $\mf{F}$ is indeed a conditional frame.
  
  Now we note that $V_{\Ax}$ defines a valuation for $\mf{F}$,
  and with this valuation the truth set of a formula in $\mb{M}_{\Ax}$ and $(\mf{F}, V_{\Ax})$
  coincides. So $(\mf{F}, V_{\Ax}), x \models \Gamma$ but
  $(\mf{F}, V_{\Ax}), x \not\models \phi$.
  This proves that $\Gamma \not\models_{\Ax} \phi$,
  as desired.
\end{proof}

  The proof of Theorem~\ref{thm:ick-Kripke-complete} highlights an important
  method for obtaining completeness: \emph{filling in} missing relations to
  turn a general conditional frame into a conditional frame.
  We will see shortly that it is not always possible to simply fill in
  missing relations with the empty relation, as this may rescind
  validity of the axioms~in~$\Ax$.

\section{The fill-in method}\label{sec:fill-in}

  The canonical model construction from Section~\ref{sec:canonical-model}
  gives completeness for any (consistent) extension of $\ick$ with respect to a
  suitable class of general conditional frames.
  To extend this to Kripke completenss, we need a way to turn a general
  conditional frame into a conditional frame that invalidates the same formulas.
  This is usually done by forgetting about the admissible sets,
  but in our case this would leave gaps in the collection of conditional relations.
  We bridge this gap via the following notion of a \emph{fill-in}.

\begin{definition}
  Let $\mb{X} = (X, \leq, s, A)$ be a general conditional frame.
  A \emph{fill-in} of $\mb{X}$ is a conditional Kripke frame
  $\mf{F} = (X, \leq, \hat{s})$ such that 
  $\hat{s}(x, a) = s(x, a)$ for all $x \in X$ and $a \in A$
\end{definition}

  In other words, the restriction of $\hat{s}$ to $X \times A$ equals $s$,
  i.e.~$\hat{s}_{\upharpoonright X \times A} = s$.
  We call this construction a fill-in because we are filling in the gaps
  in $s$ to obtain a function with domain $X \times \up(X, \leq)$.
  If $\mf{F}$ is a fill-in of $\mb{X}$, then every admissible valuation for
  $\mb{X}$ is also a valuation of $\mf{F}$. This implies that fill-ins preserve non-validity of consecutions:
  
\begin{lemma}\label{lem:inherit}
  Let $\mf{F}$ be a fill-in of a general conditional frame $\mb{X}$.
  Suppose $\mb{X}$ invalidates the consecution $\Gamma \vdash \phi$.
  Then $\mf{F}$ also invalidates $\Gamma \vdash \phi$.
\end{lemma}

\subsection{The empty fill-in}\label{subsec:empty}

  Theorem~\ref{thm:ick-Kripke-complete} used a fill-in: for any non-admissible
  upset $u$ it defined $s(x, u) = \emptyset$.
  This is called the \emph{empty fill-in}.
  
\begin{definition}
  The \emph{empty fill-in} of a general conditional frame
  $\mb{X} = (X, \leq, s, A)$ is the conditional frame
  $\kappa_{\emptyset}(\mb{X}) := (X, \leq, \kappa_{\emptyset}(s))$, where
  \begin{equation*}
    \kappa_{\emptyset}(s)
      : X \times \up(X, \leq) \to \up(X, \leq)
      : (x, u) \mapsto 
        \begin{cases}
          s(x, u) &\text{if } u \in A \\
          \emptyset &\text{otherwise}
        \end{cases}
  \end{equation*}
\end{definition}

  We have seen in Theorem~\ref{thm:ick-Kripke-complete} that
  $\kappa_{\emptyset}(\mb{X})$ is a conditional frame.
  If we want to mimic the proof of Theorem~\ref{thm:ick-Kripke-complete} to
  obtain Kripke completeness for the extension of $\ick$ with a collection
  $\Ax$ of axioms, then we need to
  take one extra step: in order for $\kappa_{\emptyset}(\can)$ to
  witness $\Gamma \not\models_{\Ax} \phi$, we need to verify that it validates $\Ax$.
  The next example shows that this need not always be the case.
  
\begin{example}\label{exm:cm-not-empty-pers}
  Consider the descriptive conditional frame
  $\mb{N}_{\infty} = (N_{\infty}, =, s, A)$ from Example~\ref{exm:descr}.
  Let $V$ be an admissible valuation
  and let $x \in N_{\infty}$ be a world that satisfies
  $p \cto q$ and $(p \wedge q) \cto r$.
  Then $s(x, V(p)) = V(p) \cap N_{\infty}^{\even} \subseteq V(q)$
  and $s(x, V(p) \cap V(q)) = V(p) \cap V(q) \cap N_{\infty}^{\even} \subseteq V(r)$.
  This implies
  \begin{equation*}
    s(x, V(p))
      = V(p) \cap N_{\infty}^{\even}
      \subseteq V(p) \cap V(q) \cap N_{\infty}^{\even}
      \subseteq V(r)
  \end{equation*}
  so that $x \models p \cto r$. So we have proven that
  $\mb{N}_{\infty}$ validates $\axCm$.
  However, we claim that the empty fill-in $\kappa_{\emptyset}(\mb{N}_{\infty})$
  does not validate $\axCm$. To see this, consider a valuation of $p, q$ and $r$
  with $V(p) = N_{\infty}$ and $V(q) = N_{\infty}^{\even}$ and $V(r) = \emptyset$.
  Then for any world $x$ we have
  $x \models p \cto q$ but $x \not\models p \cto r$, because
  \begin{equation*}
    \kappa_{\emptyset}(s)(x, V(p))
      = s(x, V(p))
      = V(p) \cap N_{\infty}^{\even}
      = N_{\infty} \cap N_{\infty}^{\even}
      = N_{\infty}^{\even},
  \end{equation*}
  which is contained in $V(q)$ but not in $V(r)$.
  Moreover, $V(p) \cap V(q) = N_{\infty}^{\even} \notin A$, so by definition of
  the empty fill-in
  $\kappa(s)(x, V(p) \cap V(q)) = \emptyset$, hence $x \models (p \wedge q) \cto r$.
  This shows that $\axCm$ is not valid in $\kappa_{\emptyset}(\mb{N}_{\infty})$.
\end{example}

  This motivates the following notion of $\kappa_{\emptyset}$-persistence.
  This plays the same role as d-persistence in intuitionistic (and classical)
  modal logic~\cite{WolZak98,BlaRijVen01}.

\begin{definition}
  A formula $\phi \in \lan{L}_{\cto}$ is called \emph{$\kappa_{\emptyset}$-persistent}
  if for every descriptive conditional frame $\mb{X}$,
  \begin{equation*}
    \mb{X} \models \phi
    \quad\text{implies}\quad
    \kappa_{\emptyset}(\mb{X}) \models \phi.
  \end{equation*}
\end{definition}

\begin{proposition}\label{prop:empty-pers-complete}
  Let $\Ax$ be a collection of $\kappa_{\emptyset}$-persistent formulas
  of $\lan{L}_{\cto}$. Then $\ick \oplus \Ax$ is complete with respect
  to the class of conditional frames that validate $\Ax$.
\end{proposition}
\begin{proof}
  The proof is the same as that of Theorem~\ref{thm:ick-Kripke-complete},
  with the additional observation that $\kappa_{\emptyset}(\can)$
  is a frame for $\ick \oplus \Ax$ because all axioms in $\Ax$
  are $\kappa_{\emptyset}$-persistent.
\end{proof}

\begin{example}
  Examples of $\kappa_{\emptyset}$-persistent axioms include
  \begin{equation*}
    p \cto p, \quad
    (p \to q) \to (p \cto q)
    \quad\text{and}\quad
    (p \cto q) \vee (q \cto p).
  \end{equation*}
  To see that $p \cto p$ is $\kappa_{\emptyset}$-persistent, let $\mb{X} = (X, \leq, s, A)$
  be a general conditional frame that validates $\mb{X}$.
  Let $x \in X$ and let $V$ be any valuation for $\kappa_{\emptyset}(\mb{X})$.
  If $V(p) \in A$ then validity of $p \cto p$ on $\mb{X}$ entails that
  $\kappa_{\emptyset}(s)(x, V(p)) = s(x, V(p)) \subseteq V(p)$, so that
  $x \models p \cto p$. If $V(p)$ is not admissible then
  $\kappa_{\emptyset}(s)(x, V(p)) = \emptyset \subseteq V(p)$,
  so again $x \models p \cto p$. Therefore $\kappa_{\emptyset}(\mb{X}) \models p \cto p$.
  The second formula can be treated similarly,
  while for the third one we use the fact that it is true at $x$ whenever either
  $V(p)$ or $V(q)$ is not admissible.
\end{example}

  Next, we exploit the close connection between (descriptive) modal frames and
  (descriptive) conditional frames to obtain a general
  $\kappa_{\emptyset}$-persistence result. We start by detailing
  the semantic connection between the logics.
  
\begin{definition}
  Let $\mf{F} = (X, \leq, s)$ be a conditional frame and $u \in \up(X, \leq)$.
  Then the modal frame $\mf{F}_{\upharpoonright u}$ is defined by
  $\mf{F}_{\upharpoonright u} := (X, \leq, R_u)$,
  where $x R_u y$ if and only if $y \in s(x, u)$.
  Similarly, if $\mb{X} = (X, \leq, s, A)$ is a descriptive
  conditional frame and $a \in A$ then
  $\mb{X}_{\upharpoonright a} := (X, \leq, R_a, A)$
  denotes the induced descriptive modal frame.
\end{definition}
  
  We can translate $\lan{L}_{\Box}$-formulas into $\lan{L}_{\cto}$-formulas
  by picking a proposition letter $p \in \Prop$ and replacing each modality
  $\Box$ by $p \cto$.

\begin{definition}
  For $p \in \Prop$, define the conditional translation $(-)^p : \lan{L}_{\Box} \to \lan{L}_{\cto}$
  recursively by:
  \begin{align*}
    q^p &= q &\text{(for $q \in {\Prop} \cup \{ \top, \bot \}$)} \\
    (\psi \star \chi)^p &= \psi^p \star \chi^p &\text{(for $\star \in \{ \wedge, \vee, \to \}$)} \\
    (\Box\psi)^p &= p \cto \psi^p
  \end{align*}
\end{definition}

  Taking $p$ to be $\top$ yields the translation used by
  Weiss~\cite[Section~4.2]{Wei19}.
  (Translations of) $\lan{L}_{\Box}$-formulas and
  (restrictions of) conditional frames and descriptive conditional frames relate as follows.

\begin{lemma}\label{lem:t1} \
  \begin{enumerate}
    \item Let $\mf{F} = (X, \leq, s)$ be a conditional frame
          and $V$ a valuation for $\mf{F}$.
          Then for any world $x \in X$ and formula $\phi \in \lan{L}_{\Box}$ we have
          $(\mf{F}, V), x \models \phi^p$ iff 
          $(\mf{F}_{\upharpoonright V(p)}, V), x \models \phi$.
    \item Let $\mb{X}$ be a descriptive conditional frame
          and $V$ an admissible valuation for $\mb{X}$.
          Then for any world $x$ and formula $\phi \in \lan{L}_{\Box}$ we have
          $(\mb{X}, V), x \models \phi^p$ iff
          $(\mb{X}_{\upharpoonright V(p)}, V), x \models \phi$.
  \end{enumerate}
\end{lemma}
\begin{proof}
  We prove the first item by induction on $\phi$.
  Since $\mf{F}$ and $\mf{F}_{\upharpoonright V(p)}$ have the same underlying
  poset, the inductive cases for propositional letters, $\bot$ $\land$,
  $\lor$ and $\to$ are immediate. For the modal case, suppose $\phi = \Box\psi$,
  so $(\Box\psi)^p = p \cto \psi$,
  and compute
  \begin{equation*}
    (\mf{F}, V), x \models p \cto \psi
      \iff s(x, V(p)) \subseteq V(\psi)
      \iff R_{V(p)}[x] \subseteq V(\psi)
      \iff (\mf{F}_{\upharpoonright V(p)}, V), x \models \Box\psi.
  \end{equation*}
  The second item can be proven analogously.
\end{proof}

  The next lemma is an analogue of Lemma~\ref{lem:t1} to validity on descriptive conditional frames.

\begin{lemma}\label{lem:t2}
  Let $\mb{X} = (X, \leq, s, A)$ be a descriptive conditional frame
  and $\phi$ a modal formula that does not contain the proposition letter
  $p$ as a subformula. Then we have
  \begin{equation*}
    \mb{X} \models \phi^p
    \iff \mb{X}_{\upharpoonright a} \models \phi \text{ for all } a \in A.
  \end{equation*}
\end{lemma}
\begin{proof}
  Suppose $\mb{X} \models \phi^p$. Let $a \in A$, $x \in X$
  and let $V$ be an admissible valuation for $\mb{X}_{\upharpoonright a}$.
  Let $V'$ be the valuation given by $V'(p) = a$
  and $V'(q) = V(q)$ for all other $q \in \Prop$.
  Since $p$ does not occur in $\phi$ and $V'(p) = a$, we get
  \begin{equation*}\label{eq:t2-1}
    (\mb{X}_{\upharpoonright a}, V), x \models \phi
      \iff (\mb{X}_{\upharpoonright a}, V'), x \models \phi
      \iff (\mb{X}_{\upharpoonright V'(p)}, V'), x \models \phi.
  \end{equation*}
  By Lemma~\ref{lem:t1} we have
  $(\mb{X}_{\upharpoonright V'(p)}, V'), x \models \phi$ if and only if
  $(\mb{X}, V'), x \models \phi^p$. Since the latter holds by assumption,
  we find $(\mb{X}_{\upharpoonright a}, V), x \models \phi$,
  and since $a$, $V$ and $x$ were arbitrary, we find $\mb{X}_{\upharpoonright a} \models \phi$
  for all $a \in A$.
  
  For the converse, assume $\mb{X}_{\upharpoonright a} \models \phi$
  for all $a \in A$. Let $V$ be any valuation for $\mb{X}$
  and $x \in X$. Then using the assumption and Lemma~\ref{lem:t1} yields
  $\mb{X} \models \phi^p$.
\end{proof}

  The following theorem allows us to deduce $\kappa_{\emptyset}$-persistence
  for certain (sets of) $\lan{L}_{\cto}$-formulas.

\begin{theorem}\label{thm:Sahl}
  Let $\phi = (\psi \to (\Box\chi \vee \xi)) \in \lan{L}_{\Box}$
  be a d-persistent modal formula
  that does not contain $p$ as a subformula.
  Then $\phi^p$ is $\kappa_{\emptyset}$-persistent.
\end{theorem}
\begin{proof}
  Let $\mb{X} = (X, \leq, s, A)$ be a descriptive conditional frame
  that validates $\phi^p$.
  Let $V$ be an arbitrary valuation for $\kappa_{\emptyset}(\mb{X})$.
  If $V(p)$ is admissible, then it follows from Lemma~\ref{lem:t2}
  that $\mb{X}_{\upharpoonright V(p)} \models \phi$.
  Since $\phi$ is d-persistent, we find that
  $\kappa(\mb{X}_{\upharpoonright V(p)}) \models \phi$,
  hence in particular $(\kappa(\mb{X}_{\upharpoonright V(p)}), V) \models \phi$.
  Now observe that $\kappa(\mb{X}_{\upharpoonright V(p)}) = (\kappa_{\emptyset}(\mb{X}))_{\upharpoonright V(p)}$, so we have
  $((\kappa_{\emptyset}(\mb{X}))_{\upharpoonright V(p)}, V) \models \phi$.
  Finally Lemma~\ref{lem:t1} implies $(\kappa_{\emptyset}(\mb{X}), V) \models \phi^p$.
  
  If $V(p)$ is not admissible, then $R_{V(p)}[x] = \emptyset$ for any $x \in X$.
  Therefore $p \cto \chi^p$ is trivially valid on $(\kappa_{\emptyset}(\mb{X}), V)$,
  hence so is $\phi^p = \psi^p \to ((p \cto \chi^p) \vee \xi^p)$.
  We conclude that $\phi$ is $\kappa_{\emptyset}$-persistent.
\end{proof}

  Combining this with Theorem~\ref{thm:d-pers-Sahl} gives:

\begin{corollary}\label{cor:Sahl}
  Let $\phi \in \lan{L}_{\Box}$ be a formula of the form
  $\phi = (\psi \to (\Box \chi \vee \xi))$.
  If $t(\phi)$ is Sahlqvist,
  then $\phi^p$ is $\kappa_{\emptyset}$-persistent.
\end{corollary}

  In particular, this captures formulas of the form $\Box\chi \vee \xi$, because these are equivalent to $\top \to (\Box\chi \vee \xi)$.
  Combining Corollary~\ref{cor:Sahl} and Proposition~\ref{prop:empty-pers-complete} entails
  completeness for a wide range of extensions of $\ick$.
  The induced completeness results all appear to be new.
  We showcase some examples.

\begin{example}
  Consider the $\lan{L}_{\Box}$-formula $q \to \Box q$.
  Its GMT-translation is $\Boxi(\Boxi q \to \Boxi\Boxm\Boxi q)$,
  and in presence of $\mix$ this is equivalent to $\Boxi(\Boxi q \to \Boxm q)$.
  Either shape of the formula is Sahlqvist, hence
  $q \to \Box q$ is d-persistent, and therefore
  $(q \to \Box q)^p = q \to (p \cto q)$
  is $\kappa_{\emptyset}$-persistent.
  
  We can use similar reasoning for $\lan{L}_{\Box}$-formulas of the form
  $\Box^nq \to \Box^mq$ where $m \geq 1$,
  and for formulas of the form $\Box^{\ell}(\Box^nq \to \Box^mq)$ where $\ell \geq 1$.
  So, the conditinoal translation of any formula of such a shape is $\kappa_{\emptyset}$-persistent.
  These include \axCf\ and \axF.
\end{example}

\begin{example}
  Consider the axiom $\Box q \vee \Box\neg q$, i.e.~$\Box q \vee \Box(q \to \bot)$.
  This translates to the axiom of \emph{conditional excluded middle} \axCem,
  $(p \cto q) \vee (p \cto \neg q)$,
  (which has turned out to be a somewhat controversial axiom in conditional
  logic, see e.g.~\cite[Section~3.4]{Lew73} and~\cite{Sta80,Wil10}).
  Its GMT-translation is:
  \begin{alignat*}{2}
    t(\Box q \vee \Box\neg q)
      &= \Boxi(\Boxi\Boxm\Boxi q \vee \Boxi\Boxm\Boxi(\Boxi q \to \Boxi\bot)) \hspace{-10em} \\
  \intertext{Since $\Boxi$ is reflexive we have $\Boxi\bot \leftrightarrow \bot$.
             Using this observation, $\mix$, and the classical duality between
             boxes and diamonds, we can rewrite this further:}
      &= \Boxi(\Boxm q \vee \Boxm(\Boxi q \to \bot))
      &&= \Boxi(\Boxm q \vee \Boxm\neg\Boxi q) \\
      &= \Boxi(\Boxm q \vee \neg\Diamondm\Boxi q)
      &&= \Boxi(\Diamondm\Boxi q \to \Boxm q)
  \end{alignat*}
  This is a Sahlqvist formula, so
  $\Box q \vee \Box\neg q$ is d-persistent and
  $(p \cto q) \vee (p \cto \neg q)$ is $\kappa_{\emptyset}$-persistent.
  Analogously, we can show $\kappa_{\emptyset}$-persistence of other versions of conditional
  excluded middle
  and excluded conditional middle:
  \begin{multicols}{2}
    \begin{enumerate}\itemindent=1.1em \itemsep=0em
      \item[(\axCemm)]  $p \cto (q \vee \neg q)$
      \item[(\axCemmm)] $q \vee (p \cto \neg(p \cto q))$
      \item[(\axEcm)]   $(p \cto q) \vee \neg(p \cto q)$
      \item[(\axEcmm)]  $(p \cto q) \vee (p \cto \neg(p \cto q))$
    \end{enumerate}
  \end{multicols}
\end{example}

\begin{example}\label{exm:other}
  All of the axioms listed in Example~\ref{exm:access-control} are
  $\kappa_{\emptyset}$-persistent. For all except $\axTrust$ this
  follows from Corollary~\ref{cor:Sahl}.
  The case for $\axTrust$ follows from the fact that validity only
  relies on the value of $s(x, a)$ for $a = \emptyset$.
  Other examples of $\kappa_{\emptyset}$-persistent include the
  following analogues of \emph{conditional linearity}:
  \begin{enumerate}\itemindent=1.1em \itemsep=0em
    \item[($\mathsf{clin_1}$)] $p \cto ((q \to r) \vee (r \to q))$
    \item[($\mathsf{clin_2}$)] $p \cto (q \to r)) \vee (p \cto (r \to q))$
    \item[($\mathsf{clin_3}$)] $(p \cto ((p \cto q) \to r)) \vee (p \cto ((p \cto r) \to q))$
  \end{enumerate}
\end{example}
  
  As a particular application of the persistence from Example~\ref{exm:other}
  we find:

\begin{theorem}
  The extension of $\ick$ with any of the axioms from
  Example~\ref{exm:access-control} is strongly Kripke complete.
\end{theorem}

\subsection{The cautious fill-in}

  We have seen in Example~\ref{exm:cm-not-empty-pers} that $\axCm$ is not
  $\kappa_{\emptyset}$-persistent. The same example can be used to show that
  $\axCt$ is not $\kappa_{\emptyset}$-persistent either.
  In this section we explore a fill-in that is specialised to accommodate the
  cautious axioms.

\begin{definition}
  The \emph{cautious fill-in} of a general conditional frame
  $\mb{X} = (X, \leq, s, A)$ is the conditional frame
  $\ckappa(\mb{X}) := (X, \leq, \ckappa(s))$, whose selection function
  is defined by
  \begin{equation*}
    \ckappa(s)
      : X \times \up(X, \leq) \to \up(X, \leq)
      : (x, u) \mapsto u \cap \bigcap \{ s(x, a) \mid a \in A \text{ and } s(x, a) \subseteq u \subseteq a \}.
  \end{equation*}
\end{definition}

\begin{lemma}
  Let $\mb{X} = (X, \leq, s, A)$ be a general conditional frame.
  Then $\ckappa(\mb{X})$ is a conditional frame.
\end{lemma}
\begin{proof}
  It is clear that $\ckappa$ is well defined, because
  $\ckappa(s)(x, u)$ is given by the intersection of upsets.
  We still need to verify that $x \leq y$ implies $s(y, u) \subseteq s(x, u)$.
  So suppose $x \leq y$.
  Then we know that $s(y, a) \subseteq s(x, a)$ for all $a \in A$.
  Now suppose $u$ is an upset that is not in $A$.
  If $a \in A$ is such that $s(x, a) \subseteq u \subseteq a$ then
  since $s(y, a) \subseteq s(x, a)$ we also have
  $s(y, a) \subseteq u \subseteq a$.
  Therefore every term in the intersection defining $s(x, u)$
  is also in the intersection defining $s(y, u)$,
  so $s(y, u) \subseteq s(x, u)$.
  If no such $a$ exists then $s(x, u) = u$ which immediately provides $s(y, u) \subseteq u = s(x, u)$.
\end{proof}

  Using the following notion of $\ckappa$-persistence, we can obtain
  completeness results for extensions of $\icc$.

\begin{definition}
  A set $\Ax \subseteq \lan{L}_{\cto}$ of axioms is called
  $\ckappa$-persistent if for any descriptive cautions frame $\mb{X}$,
  $\mb{X} \models \Ax$ implies $\ckappa(\mb{X}) \models \Ax$.
\end{definition}

\begin{proposition}\label{prop:caut-comp}
  Let $\Ax$ be a $\ckappa$-persistent set of formulas.
  Then $\icc \oplus \Ax$ is strongly complete with respect to the class
  of cautious frames that validate $\Ax$.
\end{proposition}

  By a \emph{cautious (descriptive) conditional frame} we mean a (descriptive) conditional frame
  for $\icc$, or equivalently, a conditional frame that satisfies the correspondence conditions
  \eqref{ax:refl-corr} and~\eqref{ax:sick-corr}~from Lemma~\ref{lem:sick-fr-corr}.

\begin{lemma}
  The set $\{ \axRefl, \axCt, \axCm \}$ is $\ckappa$-persistent.
\end{lemma}
\begin{proof}
  By Lemma \ref{lem:sick-fr-corr}, it suffices to prove that if a descriptive conditional
  frame $\mb{X}$ satisfies~\eqref{ax:refl-corr} and~\eqref{ax:sick-corr},
  then so does $\ckappa(\mb{X})$.
  So suppose $\mb{X}$ is such a descriptive conditional frame.
  It follows immediately from the definitions that
  $\ckappa(\mb{X})$ satisfies~\eqref{ax:refl-corr}.
  To prove that it also satisfies~\eqref{ax:sick-corr}
  we use the following claim:

\medskip\noindent
\textit{Claim: suppose $\mb{X}$ is a cautious descriptive conditional frame, $u$ is an upset and $a$ is an admissible upset.
Then $s(x, a) \subseteq u \subseteq a$ implies $\ckappa(s)(x, u) = s(x, a)$.}
  If $u \in A$ then this follows immediately from the fact that
  $\mb{X}$ satisfies~\eqref{ax:sick-corr}.
  If $u \notin A$, then it suffices to show that for any other admissible upset $b \in A$
  such that $s(x, b) \subseteq u \subseteq b$, we have
  $s(x, a) = s(x, b)$.
  Given such $a, b$, it follows from the assumption and~\eqref{ax:refl-corr} that
  $s(x, a) \subseteq a \cap b \subseteq a$ and $s(x, b) \subseteq a \cap b \subseteq b$.
  Since $a \cap b \in A$,~\eqref{ax:sick-corr} then gives
    $s(x, a) = s(x, a \cap b) = s(x, b)$.
  
  \medskip\noindent
  Next we prove that $\ckappa(\mb{X})$ satisfies~\eqref{ax:sick-corr}.
  To this end, let $x \in X$ and let $u, v$ be upsets such that $\ckappa(s)(x, u) \subseteq v \subseteq u$.
  If $u \in A$ then it follows immediately from the claim that
  $\ckappa(s)(x, u) = \ckappa(s)(x, v)$.
  Now suppose $u \notin A$.
    If $s(x, a) \subseteq u \subseteq a$ for some $a \in A$
    then
    \begin{equation*}
      s(x, a) = \ckappa(s)(x, u) \subseteq v \subseteq u \subseteq a,
    \end{equation*}
    so by the claim $\ckappa(s)(x, v) = s(x, a) = \ckappa(s)(x, u)$.
    If no such $a \in A$ exists, then $\ckappa(s)(x, u) = u$ by definition,
    so that the assumption $\ckappa(s)(x, u) \subseteq v \subseteq u$
    implies that $u = v$, hence $\ckappa(s)(x, u) = \ckappa(s)(x, v)$.
\end{proof}

\begin{lemma}
  The following axioms are $\ckappa$-persistent:
  \begin{multicols}{3}
  \begin{enumerate}\itemindent=1.1em \itemsep=0em
    \item[(\axRefl)] $p \cto p$
    \item[(\axMP)] $(p \cto q) \to (p \to q)$
    \item[(\axCS)] $(p \wedge q) \to (p \cto q)$
  \end{enumerate}
  \end{multicols}
\end{lemma}
\begin{proof}
  We have seen that $p \cto p$ is $\ckappa$-persistent.
  For $\axMP$, we use the fact that validity in conditional frames and descriptive
  conditional frames corresponds to ${\uparrow}x \cap a \subseteq s(x, a)$ for all
  worlds $x$ and (admissible) upsets $a$. (The routine proof of this fact is provided in the appendix.)
  Suppose a descriptive conditional frame $\mb{X} = (X, \leq, s, A)$ satisfies this
  condition. Then we need to prove that for any non-admissible upset $u$ we have
  ${\uparrow}x \cap u \subseteq \ckappa(s)(x, u)$.
  Note that for each $a \in A$ such that $s(x, a) \subseteq u \subseteq a$ we have
  \begin{equation*}
    {\uparrow}x \cap u
      \subseteq {\uparrow}x \cap a
      \subseteq s(x, a)
  \end{equation*}
  Moreover, we clearly have ${\uparrow}x \cap u \subseteq u$.
  Therefore, by definition of $\ckappa(s)(x, u)$ we get
  ${\uparrow}x \cap u \subseteq \ckappa(s)(x, u)$.
  
  The axiom $\axCS$ is valid in a (descriptive) conditional frame if and only if for all worlds $x$ and (admissible) upsets $a$,
  $x \in a$ implies $s(x, a) \subseteq {\uparrow}x$ (Lemma~\ref{lem:cs-corr}). Suppose $\mb{X} = (X, \leq, s, A)$ is a descriptive
  frame satisfying this condition. We aim to show that it holds for non-admissible
  upsets too.
  So suppose $u$ is a non-admissible upset and $x \in u$.
  Then we have
  \begin{equation*}
    s(x, X)
      \subseteq {\uparrow}x
      \subseteq u
      \subseteq X
  \end{equation*}
  so by definition $\ckappa(s)(x, u) \subseteq s(x, X) \subseteq {\uparrow}x$,
  as desired.
\end{proof}

  These persistence lemmas give rise to several more completeness results
  for extensions of $\ick$, all of which are new in the
  intuitionistic conditional logic literature.

\begin{theorem}
  Let $\Ax \subseteq \{ \axRefl, \axMP, \axCS \}$.
  Then $\ick \oplus \Ax$ and $\icc \oplus \Ax$ are strongly complete
  with respect to the classes of conditional frames validating them, respectively.
\end{theorem}
\begin{proof}
  This follows from Proposition~\ref{prop:caut-comp} and the
  preceding two lemmas.
\end{proof}

  In light of the general $\kappa_{\emptyset}$-persistence result
  from Section~\ref{subsec:empty}, it is natural to pose the following question:

\begin{question}
  Is it possible to derive general $\ckappa$-persistence result, similar to that in Section~\ref{subsec:empty}?
\end{question}

\subsection{Outlook}

For axioms that are not $\kappa_{\emptyset}$- or $\ckappa$-persistent not all
hope is lost: the fill-in technique can be applied with other types of fill-ins.
More involved conditional logics may require more complicated fill-in
constructions.
Several of these suggest themselves, each suited to a different class of axioms.

\medskip
\noindent\textit{Fill-in agnostic formulas.} \;
A type of persistence we have not yet mentioned is that of \emph{fill-in agnostic} formulas. The validity of these formulas on a conditional frame is independent of the choice of fill-in altogether. This occurs when the conditional antecedent is only ever evaluated at admissible sets, so that every fill-in agrees on its truth value. Prototypical example include $\axTrust$ and $(\top \cto p) \to p$.
In these cases, the interpretation of the antecedent of the conditional implication is either $\bot$ or $\top$, hence its truth set is always admissible, hence they are persistent with repsect to any notion of fill-in.

Another fill-in agnostic formula is the law of excluded middle, $p \vee \neg p$. As a consequence, all completeness results derived in this paper also apply to the classical setting.

\medskip
\noindent\textit{The principal fill-in} sets $\kappa_{\mathrm{pr}}(s)(x, u) = {\uparrow}x$ for all non-admissible $u$. Formulas preserved by this fill-in should be those where the conditional relation agrees with $\leq$, at least locally.
This may lead to general persistence results akin the one in Section~\ref{subsec:empty}.

\medskip
\noindent\textit{The reflexive fill-in} sets $\kappa_{\mathrm{ref}}(s)(x, u) = u$ for all non-admissible $u$. This fill-in is always well defined and produces a conditional frame in which every world accesses, via a non-admissible proposition, exactly the worlds satisfying that proposition. While a full persistence theory for this fill-in remains to be developed, it appears well suited to axioms that assert a close relationship between the selection function output and its propositional input.

\medskip
\noindent\textit{The monotone fill-in} sets $\kappa_{\mathrm{mon}}(s)(x, u) = \bigcup\{s(x, a) \mid a \in A,\, a \subseteq u\}$. This fill-in is natural in the presence of monotonicity axioms, which assert that the selection function is monotone in its second argument. The monotone fill-in is in a natural sense the least fill-in that makes the selection function monotone on all upsets, and one would expect persistence results for monotone axioms to follow from this construction.

\medskip
A general investigation of fill-ins and their associated persistence notions would clarify the landscape of Kripke completeness for intuitionistic conditional logics. In particular, it would be interesting to determine whether there exist axioms for which no fill-in derived from a modal translation can witness persistence, but for which a persistence result can nonetheless be obtained by a direct conditional argument. Such results would demonstrate that the conditional setting is genuinely richer than its modal counterpart from the perspective of completeness theory.

\section{Conclusion}

  We introduced the fill-in method to obtain completeness results
  for intuitionistic conditional logics.
  More specifically, we used a canonical model construction to prove completeness
  of extensions of the basic intuitionistic conditional logic with respect to
  certain general conditional frames,
  and then employed fill-ins to close the gap between the general and
  full semantics of the logic.

  There are many directions further development of intuitionistic conditional
  logics.

\begin{description}
  \item[\bfseries Explore more fill-ins and persistence results]
        The most obvious extension to our paper is to further explore
        fill-ins and the completeness results they prove. We have only
        introduced a few notions of fill-ins, but there could be others
        that work well for more complicated axioms or allow us to
        prove completeness for combinations of axioms that are not captured
        by the fill-ins we have seen so far.
        In particular, we expect that mirroring the canonical
        extension~\cite{GehJon94,GehJon04}
        will give rise to a fill-in that is well suited for intuitionistic
        conditional logics satisfying monotonicity axioms.

  \item[\bfseries More transfer results]
        The GMT translation not only gives rise to persistence results,
        but also to theorems regarding decidability and the finite model
        property.
        It would be interesting to investigate to what extend these can be
        transferred to the conditional setting.
  \item[\bfseries Conditional possibility operator]
        Lastly, we note that our conditional operator is based on a modal
        necessity operator $\cto$. One could consider a conditional possibility
        operator $\Diamondright$ in the spirit of Lewis~\cite{Lew73} with the interesting
        caveat that the intuitionistic diamond is often not considered dual to the
        intuitionistic box unlike classical modal logic~\cite{Olk24,DalGir26}.
        It would be interesting to investigate the fill-in method in presence of $\Diamondright$.
\end{description}

\bibliographystyle{plainnat}
{\footnotesize
\bibliography{ick.bib}}

@article{Aba07,
  author    = {Abadi, M.},
  title     = {Access control in a core calculus of dependency},
  journal   = {Electronic Notes in Theoretical Computer Science, Computation, Meaning, and Logic: Articles dedicated to {G}ordon {P}lotkin},
  volume    = {172},
  pages     = {5--31},
  year      = {2007},
}

@inproceedings{Aba08,
  author    = {M. Abadi},
  title     = {Variations in access control logic},
  booktitle = {Deontic Logic in Computer Science (DEON 2008)},
  editor    = {R. van der Meyden and 
               L. van der Torre},
  pages     = {96--109},
  OTPseries = {Lecture Notes in Computer Science},
  OPTvolume = {5076},
  publisher = {Springer},
  address   = {Berlin, Heidelberg},
  year      = {2008},
}

@article{AbaBurLamPlo93,
  author    = {M. Abadi and
               M. Burrows and
               B. Lampson and
               G. Plotkin},
  title     = {A Calculus for Access Control in Distributed Systems},
  journal   = {{ACM} Transactions on Programming Languages and Systems},
  volume    = {15},
  number    = {4},
  year      = {1993},
  pages     = {706--734},
}

@article{Ada65,
  author    = {E. W. Adams},
  title     = {The logic of conditionals},
  journal   = {Inquiry},
  volume    = {8},
  issue     = {1–4},
  pages     = {166--197},
  year      = {1965},
  doi       = {10.1080/00201746508601430},
}

@book{Ada75,
  author    = {E. W. Adams},
  title     = {The Logic of Conditionals},
  subtitle  = {An Application of Probability to Deductive Logic},
  doi       = {10.1007/978-94-015-7622-2},
  publisher = {Springer},
  address   = {Dordrecht},
  year      = {1975},
}

@article{Bad21,
  author    = {C. Badura},
  title     = {More aboutness in imagination},
  journal   = {Journal of Philosophical Logic},
  volume    = {50},
  number    = {3},
  pages     = {523--547},
  year      = {2021},
  doi       = {10.1007/s10992-020-09575-4},
}

@article{BalCin18,
  author    = {Baltag, A. and
                Cin{\`a}, G.},
  title     = {Bisimulation for Conditional Modalities},
  journal   = {Studia Logica},
  volume    = {106},
  pages     = {1--33},
  year      = {2018},
  doi       = {10.1007/s11225-017-9723-2},
}

@book{Ben03,
  author    = {Bennett, J.},
  title     = {A Philosophical Guide to Conditionals},
  publisher = {Oxford University Press},
  address   = {Oxford},
  year      = {2003},
  doi       = {10.1093/0199258872.001.0001},
}

@inproceedings{Ber18b,
  author    = {F. Berto},
  title     = {The theory of topic-sensitive intentional modals},
  editor    = {I. Sedl\'{a}r and
               M. Blicha},
  booktitle = {The Logica Yearbook 2018},
  pages     = {31--56},
  publisher = {College Publications},
  address   = {Rickmansworth},
  year      = {2018},
}

@book{Ber22,
  author    = {Berto, F. and
               Hawke, P and
               {\"O}zg{\"u}n, A},
  title     = {Topics of Thought: The Logic of Knowledge, Belief, Imagination},
  publisher = {Oxford University Press},
  year      = {2022},
  OPTisbn   = {9780192857491},
  doi       = {10.1093/oso/9780192857491.001.0001},
  address   = {Oxford},
}

@book{BlaRijVen01,
  author    = {Blackburn, P. and
                Rijke, M. de and
                Venema, Y.},
  title     = {Modal Logic},
  OPTseries    = {Cambridge Tracts in Theoretical Computer Science},
  year      = {2001},
  publisher = {Cambridge University Press},
  address   = {Cambridge},
  doi       = {10.1017/CBO9781107050884},
}

@article{BoeGabGenTor09,
  author    = {G. Boella and
               D. M. Gabbay and
               V. Genovese and
               L. van der Torre},
  title     = {Fibred Security Language},
  journal   = {Studia Logica},
  volume    = {92},
  pages     = {395--436},
  year      = {2009},
  doi       = {10.1007/s11225-009-9201-6},
}

@inproceedings{Bou90,
  author    = {C. Boutilier},
  title     = {Conditional Logics of Normality as Modal Systems},
  booktitle = {Proc.~{AAAI} 1990},
  pages     = {594--599},
  year      = {1990},
}

@article{CiaLiu20,
  author    = {Ciardelli, I. and
                Liu, X.},
  title     = {Intuitionistic Conditional Logics},
  journal   = {Journal of Philosophical Logic},
  year      = {2020},
  volume    = {49},
  pages     = {807--832}, 
  OPTissn   = {1573-0433},
  doi       = {10.1007/s10992-019-09538-4},
}

@article{Che75,
  author    = {B. F. Chellas},
  title     = {Basic Conditional Logic},
  journal   = {Journal of Philosophical Logic},
  volume    = {4},
  number    = {2},
  pages     = {133--153},
  publisher = {Springer},
  year      = {1975},
  note      = {jstor:\href{http://www.jstor.org/stable/30226114}{30226114}},
}

@book{Che80,
  author    = {Chellas, B. F.},
  title     = {Modal Logic: An Introduction},
  year      = {1980},
  publisher = {Cambridge University Press},
  address   = {Cambridge},
  OPTisbn   = {9780521295154},
  OPTlccn   = {76047197},
}

@article{ConGorVak06,
  author    = {W. Conradie and
               V. Goranko and
               D. Vakarelov},
  title     = {Algorithmic correspondence and completeness in modal logic. {I}. {T}he core algorithm {SQEMA}},
  journal   = {Logical Methods in Computer Science},
  volume    = {2},
  number    = {1},
  year      = {2006},
}

@inproceedings{CroLam92,
  author    = {Crocco, G. and
                Lamarre, P.},
  title     = {On the Connection Between Non-monotonic Inference Systems and Conditional Logics},
  booktitle = {Proceedings of the Third International Conference on Principles of
                Knowledge Representation and Reasoning},
  location  = {Cambridge, MA},
  OPTseries    = {KR'92},
  pages     = {565--571},
  year      = {1992},
  publisher = {Morgan Kaufmann Publishers Inc.},
  address   = {San Francisco, CA, USA},
  OPTisbn   = {1-55860-262-3},
  OPTacmid  = {3087281},
}

@inproceedings{DalGir26,
  title     = {A Proof-Theoretic View of Basic Intuitionistic Conditional Logic},
  author    = {T. Dalmonte and
               M. Girlando},
  booktitle = {Proceedings {TABLEAUX} 2025},
  editor    = {Pozzato, G. L. and
               Uustalu, T},
  OPTseries = {Lecture Notes in Computer Science},
  OPTvolume = {15980},
  publisher = {Springer},
  address   = {Cham},
  pages     = {354--373},
  year      = {2026},
  doi       = {10.1007/978-3-032-06085-3_19},
}

@inproceedings{DeT02,
  author    = {J. DeTreville},
  title     = {Binder, a logic-based security language},
  booktitle = {Proceedings of IEEE Symposium on Security and Privacy},
  pages     = {105--113},
  year      = {2002},
}

@article{Edg95,
  author    = {Edgington, D.},
  title     = {On Conditionals},
  journal   = {Mind},
  volume    = {104},
  number    = {414},
  pages     = {235--329},
  year      = {1995},
  doi       = {10.1093/mind/104.414.235},
}

@article{EgrRosSpr21a,
  author    = {\'Egr\'e, P. and
               Rossi, L. and
               Sprenger, J.},
  title     = {De {F}inettian Logics of Indicative Conditionals Part {I}:
               Trivalent Semantics and Validity},
  journal   = {Journal of Philosophical Logic},
  volume    = {50},
  year      = {2021},
  pages     = {187--213},
  doi       = {10.1007/s10992-020-09549-6},
}

@article{EgrRosSpr21b,
  author    = {\'Egr\'e, P. and
               Rossi, L. and
               Sprenger J.},
  title     = {De {F}inettian Logics of Indicative Conditionals Part {II}:
               Proof Theory and Algebraic Semantics},
  journal   = {Journal of Philosophical Logic},
  volume    = {50},
  pages     = {215--247},
  year      = {2021},
  doi       = {10.1007/s10992-020-09572-7},
}

@article{Esa74,
  author    = {L. L. Esakia},
  title     = {Topological {K}ripke models},
  journal   = {Soviet Mathematics Doklady},
  volume    = {15},
  pages     = {147--151},
  year      = {1974},
}

@book{Esa19,
  author    = {Esakia, L. L.},
  title     = {Heyting Algebras},
  subtitle  = {Duality Theory},
  editors   = {G. Bezhanishvili and
                W. H. Holliday},
  note      = {Translated by A. Evseev},
  series    = {Trends in Logic},
  year      = {2019},
  publisher = {Springer},
  address   = {Springer},
  OPTisbn   = {978-3-030-12095-5},
  doi       = {10.1007/978-3-030-12096-2},
}

@inproceedings{Fin36,
  author    = {Finetti, B. de},
  year      = {1936},
  title     = {La Logique de La Probabilit\'e},
  booktitle = {Actes Du Congr\`{e}s International de Philosophie Scientifique, Paris 1935, Vol. IV: Induction et Probabilité},
  address   = {Paris},
  publisher = {Hermann},
  pages     = {31--39},
}

@InProceedings{Gab85,
  author    = {Gabbay, D. M.},
  title     = {Theoretical Foundations for Non-Monotonic Reasoning in Expert Systems},
  booktitle = {Logics and Models of Concurrent Systems},
  pages     = {439--457},
  year      = {1989},
  OPTeditor    = {Apt, K. R.},
  publisher = {Springer-Verlag},
  address   = {Berlin, Heidelberg},
  OPTisbn      = {978-3-642-82453-1},
}

@inproceedings{GarPfe06,
  author    = {Garg, D. and
               Pfenning, F.},
  title     = {Non-interference in constructive authorization logic},
  booktitle = {Proceedings of the 19th IEEE Computer Security Foundations Workshop (CSFW
19)},
  pages     = {283--296},
  year      = {2006},
}

@inproceedings{GarAba08,
  author    = {Garg, D. and
               Abadi, M.}, 
  title     = {A Modal Deconstruction of Access Control Logics},
  editor    = {Amadio, R.},
  booktitle = {Foundations of Software Science and Computational Structures (FoSSaCS 2008)}, 
  OPTseries = {Lecture Notes in Computer Science},
  OPTvolume = {4962},
  pages     = {216--230},
  publisher = {Springer},
  address   = {Berlin, Heidelberg},
  doi       = {10.1007/978-3-540-78499-9_16},
  year      = {2008},
}

@article {GehJon94,
  author    = {Gehrke, M. and
               J\'onsson, B.},
  title     = {Bounded distributive lattices with operators},
  journal   = {Mathematica Japonica},
  volume    = {40},
  number    = {2},
  year      = {1994},
  pages     = {207--215},
}

@article{GehJon04,
  author    = {M. Gehrke and
                B. J\'{o}nsson},
  title     = {Bounded distributive lattice expansions},
  journal   = {Mathematica Scandinavica},
  volume    = {94},
  number    = {1},
  pages     = {13--45},
  year      = {2004},
  note      = {jstor:\href{https://www.jstor.org/stable/24493402}{24493402}},
}

@article{GenGioGliPoz14,
  author    = {V. Genovese and
               L. Giordano and
               V. Gliozzi and
               G. L. Pozzato},
  title     = {Logics in access control: a conditional approach},
  journal   = {Journal of Logic and Computation},
  volume    = {24},
  issue     = {4},
  year      = {2014},
  pages     = {705--762},
  doi       = {10.1093/logcom/exs040},
}

@phdthesis{Gir19,
  author    = {M. Girlando},
  title     = {On the Proof Theory of Conditional Logics},
  year      = {2019},
  school    = {{A}ix-{M}arseille {U}niversit\'{e}
                and {U}niversity of {H}elsinki},
}

@InProceedings{GirNegSba19,
  author    = {Girlando, M. and
                Negri, S. and
                Sbardolini, G.},
  title     = {Uniform Labelled Calculi for Conditional and Counterfactual Logics},
  booktitle = {Proc.~{WoLLIC} 2019},
  OPTeditor    = {Iemhoff, R. and
                Moortgat, M. and
                {de Queiroz}, R.},
  pages     = {248--263},
  year      = {2019},
  publisher = {Springer},
  address   = {Berlin, Heidelberg},
  OPTisbn   = {978-3-662-59533-6},
  doi       = {10.1007/978-3-662-59533-6_16},
}

@article{GriNegOli21,
  author    = {M. Girlando and
               S. Negri and
               N. Olivetti},
  title     = {Uniform labelled calculi for preferential conditional logics
               based on neighbourhood semantics},
  journal   = {Journal of Logic and Computation},
  volume    = {31},
  issue     = {3},
  year      = {2021},
  pages     = {947--997},
  doi       = {10.1093/logcom/exab019},
}

@phdthesis{Gro22-phd,
  author    = {Groot, J. de},
  title     = {Dualities in Modal Logic},
  year      = {2022},
  school    = {The Australian National University},
}

@article{GuzSqu90,
  author    = {Guzm\'{a}n, F. and
               Squier, C. C.},
  title     = {The algebra of conditional logic},
  journal   = {Algebra Universalis},
  volume    = {27},
  pages     = {88--110},
  year      = {1990},
  doi       = {10.1007/BF01190256},
}

@Inproceedings{HawOzg23,
  author    = {Hawke, P. and
               {\"O}zg{\"u}n, A.},
  title     = {Truthmaker Semantics for Epistemic Logic},
	editor="Faroldi, Federico L. G.
	and Van De Putte, F.",
	bookTitle="Kit Fine on Truthmakers, Relevance, and Non-classical Logic",
	year="2023",
	publisher="Springer International Publishing",
	address="Cham",
	pages="295--335",
	doi={10.1007/978-3-031-29415-0_15},
}

@book{Kra99,
  author    = {Kracht, M.},
  title     = {Tools and Techniques in Modal Logic},
  year      = {1999},
  publisher = {Elsevier},
}

@article{KraLehMag90,
  author    = {Kraus, S. and
               Lehmann, D. and
               Magidor, M.},
  title     = {Nonmonotonic reasoning, preferential models and cumulative logics},
  journal   = {Artificial Intelligence},
  volume    = {44},
  pages     = {167--207},
  year      = {1990},
}

@article{Lew12,
  author    = {C. I. Lewis},
  title     = {Implication and the Algebra of Logic},
  journal   = {Mind},
  volume    = {21},
  number    = {84},
  pages     = {522--531},
  doi       = {10.1093/mind/XXI.84.522},
  year      = {1912},
}

@article{Lew13,
  author    = {C. I. Lewis},
  title     = {A New Algebra of Implications and Some Consequences},
  journal   = {The Journal of Philosophy, Psychology and Scientific Methods},
  volume    = {10},
  number    = {16},
  year      = {1913},
  pages     = {428--438},
  doi       = {10.2307/2012900},
}

@article{Lew14,
  author    = {C. I. Lewis},
  title     = {The matrix algebra for implications},
  journal   = {The Journal of Philosophy, Psychology and Scientific Methods},
  volume    = {11},
  number    = {22},
  year      = {1914},
  pages     = {589--600},
  doi       = {10.2307/2012652},
}

@book{Lew18,
  author    = {C. I. Lewis},
  title     = {A Survey of Symbolic Logic},
  publisher = {University of California Press},
  address   = {California},
  year      = {1918},
}

@book{Lew73,
  author    = {Lewis, D.},
  title     = {Counterfactuals},
  year      = {1973},
  publisher = {Harvard University Press},
  address   = {Harvard},
}

@article{LiGroFei03,
  author    = {N. Li and
               B. N. Grosof and
               J. Feigenbaum},
  title     = {Delegation logic: a logic-based approach to distributed authorization},
  journal   = {{ACM} Transactions on Information and System Security},
  volume    = {6},
  pages     = {128--171},
  year      = {2003},
}

@article{LitVis18,
  author    = {Litak, T. and
                 Visser, A.},
  title     = {{L}ewis meets {B}rouwer: Constructive strict implication},
  journal   = {Indagationes Mathematicae},
  volume    = {29},
  number    = {1},
  pages     = {36--90},
  year      = {2018},
  doi       = {10.1016/j.indag.2017.10.003},
}

@inproceedings{LitVis24,
  author    = {Litak, T. and
               Visser, A.},
  title     = {{L}ewisian Fixed Points {I}: Two Incomparable Constructions},
  editor    = {Bezhanishvili, N. and
               Iemhoff, R. and
               Yang, F.},
  booktitle = {Dick de Jongh on Intuitionistic and Provability Logics},
  series    = {Outstanding Contributions to Logic},
  volume    = {28},
  publisher = {Springer},
  address   = {Cham},
  doi       = {10.1007/978-3-031-47921-2_2},
  year      = {2024},
}

@mastersthesis{Liu18,
  title     = {Intuitionistic Lewisian Conditional Logics},
  author    = {Xinghan Liu},
  school    = {Fakult\"at f\"ur Philosophie},
  institute = {Ludwig–Maximilians–Universit\"at M\"unchen},
  year      = {2018},
  url       = {https://www.researchgate.net/publication/331438337_Intuitionistic_Lewisian_Conditional_Logics},
}

@article{Mac08,
  author    = {H. MacColl},
  title     = {`{I}f' and `{I}mply'},
  journal   = {Mind},
  volume    = {17},
  number    = {65},
  year      = {1908},
  pages     = {151--152},
}

@incollection{Mos24,
  author    = {Moschovakis, J.},
  title     = {Intuitionistic Logic},
  booktitle = {The {Stanford} Encyclopedia of Philosophy},
  editor    = {E. N. Zalta and U. Nodelman},
  year      = {2024},
  OPTedition = {{S}ummer 2024},
  publisher = {Metaphysics Research Lab, Stanford University},
  address   = {Stanford},
  url       = {https://plato.stanford.edu/archives/sum2024/entries/logic-intuitionistic/},
}

@article{Nut75,
  author    = {D. Nute},
  title     = {Counterfactuals},
  journal   = {Notre Dame Journal of Formal Logic},
  volume    = {16},
  pages     = {476--482},
  year      = {1975},
}

@article{Nut79,
  author    = {D. Nute},
  title     = {Algebraic Semantics for Conditional Logics},
  journal   = {Reports on Mathematical Logic},
  volume    = {10},
  pages     = {79--101},
  year      = {1979},
}

@article{Nut88,
  author    = {D. Nute},
  title     = {Topics in Conditional Logic},
  journal   = {Studia Logica},
  volume    = {47},
  number    = {2},
  pages     = {175--176},
  publisher = {Springer},
  year      = {1988}
}

@article{Olk24,
  author    = {Olkhovikov, G.},
  title     = {An Intuitionistically Complete System of Basic Intuitionistic Conditional Logic},
  journal   = {Journal of Philosophical Logic},
  volume    = {53},
  pages     = {1199--1240},
  year      = {2024},
  doi       = {10.1007/s10992-024-09763-6},
}

@article{OzgCot25,
  author    = {A. \"{O}zg\"{u}n and
	          A. J. Cotnoir},
  title     = {Imagination, Mereotopology, and Topic Expansion},
  journal   = {The Review of Symbolic Logic},
  year      = {forthcoming}
}

@inproceedings{PfeKle10,
  author    = {Pfeifer, N. and
               Kleiter, G. D.},
  title     = {Conditionals and probability},
  editor    = {M. Oaksford and N. Chater},
  booktitle = {The conditional in mental probability logic},
  pages     = {153--173},
  publisher = {Oxford University Press},
  address   = {Oxford},
  year      = {2010},
}

@incollection{Ram31,
  author    = {Ramsey, F. P.},
  title     = {General Propositions and Causality},
  booktitle = {The Foundations of Mathematics and other Logical Essays},
  editor    = {R. B. Braithwaite},
  address   = {London},
  publisher = {Kegan Paul},
  pages     = {237--255},
  year      = {1931},
}

@book{Rei35,
  author    = {Reichenbach, H.},
  year      = {1935},
  title     = {Wahrscheinlichkeitslehre},
  address   = {Leiden},
  publisher = {Sijthoff},
}

@article{Seg89,
  author    = {Segerberg, K.},
  title     = {Notes on conditional logic},
  journal   = {Studia Logica},
  volume    = {48},
  pages     = {157--168},
  year      = {1989},
  doi       = {10.1007/BF02770509},
}

@incollection{Stal68,
  author    = {R. C. Stalnaker},
  title     = {A Theory of Conditionals},
  booktitle = {IFS},
  editor    = {Harper, W. L. and 
               Stalnaker, R. and
               Pearce, G.},
  series    = {The University of Western Ontario Series in Philosophy of Science},
  volume    = {15},
  publisher = {Springer},
  address   = {Dordrecht},
  year      = {1968},
  doi       = {10.1007/978-94-009-9117-0_2},
}

@incollection{Sta80,
  author    = {R. C. Stalnaker},
  title     = {A Defense of Conditional Excluded Middle},
  booktitle = {Ifs},
  subtitle  = {Conditionals, Belief, Decision, Chance and Time},
  editor    = {W. L. Harper and
               R. Stalnaker and
               G. Pearce},
  year      = {1980},
  pages     = {87--104},
  doi       = {10.1007/978-94-009-9117-0_4},
  publisher = {D. Reidel Publishing Company},
  address   = {Dordrecht, Holland},
}

@article{Ste70,
  author    = {C. L. Stevenson},
  title     = {If-{I}culties},
  journal   = {Philosophy of Science},
  year      = {1970},
  volume    = {37},
  number    = {1},
  pages     = {27--49},
}

@book{Unt13,
  author    = {M. Unterhuber},
  title     = {Possible Worlds Semantics for Indicative and Counterfactual Conditionals?},
  subtitle  = {A Formal Philosophical Inquiry into Chellas-Segerberg Semantics},
  publisher = {De Gruyter},
  OPTseries = {Logos},
  OPTvolume = {21},
  address   = {Frankfrut},
  doi       = {10.1515/9783110323665},
  year      = {2013},
}

@article{UntSch14,
  author    = {Unterhuber, M. and
                Schurz, G.},
  title     = {Completeness and Correspondence in {C}hellas-{S}egerberg Semantics},
  journal   = {Studia Logica},
  volume    = {102},
  pages     = {891--911},
  year      = {2014},
  doi       = {10.1007/s11225-013-9504-5},
}

@inproceedings{Fra76,
  author    = {Van Fraassen, B. C.},
  title     = {Probabilities of Conditionals},
  booktitle = {Foundations of Probability Theory, Statistical Inference, and Statistical Theories of Science},
  editor    = {Harper, W. L. and
               Hooker, C. A.},
  OPTseries = {The University of Western Ontario Series in Philosophy of Science},
  OPTvolume = {6a},
  publisher = {Springer},
  address   = {Dordrecht},
  doi       = {10.1007/978-94-010-1853-1_10},
  year      = {1976},
}

@phdthesis{Vel85,
  title     = {Logic for conditionals},
  author    = {Veltman, F.},
  school    = {University of Amsterdam},
  year      = {1985},
}

@article{WanUnt19,
  author    = {Wansing, H. and
               Unterhuber, M.}, 
  title     = {Connexive conditional logic. Part {I}},
  journal   = {Logic and Logical Philosophy},
  volume    = {28},
  pages     = {567--610},
  year      = {2019},
}

@article{Wei19,
  author    = {Weiss, Y.},
  title     = {Basic Intuitionistic Conditional Logic},
  journal   = {Journal of Philosophical Logic},
  volume    = {48},
  number    = {3},
  pages     = {447-469},
  year      = {2019},
  doi       = {10.1007/s10992-018-9471-4}
}

@phdthesis{Wei19b,
  author    = {Weiss, Y.},
  title     = {Frontiers of conditional logic},
  year      = {2019},
  school    = {City University of {N}ew {Y}ork},
}

@article{Wil10,
  author    = {J. R. G. Williams},
  title     = {Defending Conditional Excluded Middle},
  journal   = {No{\^u}s},
  volume    = {44},
  issue     = {4},
  year      = {2010},
  pages     = {650--668},
}

@article{WolZak97,
  author    = {Wolter, F. and
                Zakharyaschev, M.},
  title     = {The relation between intuitionistic and classical modal logics},
  journal   = {Algebra and Logic},
  volume    = {36},
  number    = {2},
  pages     = {73--92},
  year      = {1997},
  doi       = {10.1007/BF02672476},
}

@incollection{WolZak98,
  author    = {Wolter, F. and
                Zakharyaschev, M.},
  title     = {Intuitionistic Modal Logics as Fragments of Classical Bimodal Logics},
  booktitle = {Logic at Work, Essays in honour of {H}elena {R}asiowa},
  editor    = {E. Orlowska},
  pages     = {168--186},
  year      = {1998},
  publisher = {Springer--Verlag},
}

@inproceedings{WolZak99,
  author    = {Wolter, F. and
                Zakharyaschev, M.},
  title     = {Intuitionistic Modal Logic},
  bookTitle = {Logic and Foundations of Mathematics: Selected Contributed Papers of
                the Tenth International Congress of Logic, Methodology and Philosophy
                of Science},
  editor    = {Cantini, A. and
                Casari, E. and
                Minari, P.},
  pages     = {227--238},
  year      = {1999},
  publisher = {Springer},
  address   = {Dordrecht, Netherlands},
  OPTisbn   = {978-94-017-2109-7},
  doi       = {10.1007/978-94-017-2109-7_17},
}

@article{Ferguson_2024,
  author={Ferguson, T. M.}, 
  title={Subject-Matter and Intensional Operators {III}: State-Sensitive Subject-Matter and Topic Sufficiency},
  journal={The Review of Symbolic Logic}, 
  volume={17},
  number={4},
  DOI={10.1017/S1755020323000230},
  year={2024}, 
  pages={1070–1096}
}

@article{OzgunBerto_2021,
  author={{\"O}zg{\"u}n, A and Berto, F.},
  title={Dynamic Hyperintensional Belief Revision},
  journal={The Review of Symbolic Logic},
  volume={14},
  number={3},
  DOI={10.1017/S1755020319000686},
  year={2021},
  pages={766–811}
}

@inproceedings{Dum84,
  address={Cambridge},
  title={The philosophical basis of intuitionistic logic}, 
  booktitle={Philosophy of Mathematics: Selected Readings},
  publisher={Cambridge University Press},
  author={Dummett, M.},
  editor={Benacerraf, P. and Putnam, H.},
  year={1984},
  pages={97–129}}

\appendix
\section{Some correspondence conditions}

\begin{lemma}
  Let $\mb{X} = (X, \leq, s, A)$ be a full or descriptive general conditional frame.
  Then $\mb{X}$ validates $\axMP$ $(p \cto q) \to (p \to q)$
  if and only if it satisfies for all $x \in X$ and $a \in A$:
  \begin{enumerate}\itemindent=3em
    \renewcommand{\labelenumi}{\textup{(}\theenumi\textup{)} }
    \renewcommand{\theenumi}{$\mathsf{mp}$-corr}
    \item \label{ax:mp-corr} ${\uparrow}x \cap a \subseteq s(x, a)$
  \end{enumerate}
\end{lemma}
\begin{proof}
  Suppose $\mb{X}$ satisfies~\eqref{ax:mp-corr}.
  Let $V$ be any valuation, and let $x \in X$ be such that $x \models (p \cto q)$.
  Then $s(x, V(p)) \subseteq V(q)$.
  Moreover, by~\eqref{ax:mp-corr} we have
  ${\uparrow}x \cap V(p) \subseteq s(x, V(p))$.
  This implies ${\uparrow}x \cap V(p) \subseteq V(q)$,
  hence $x \models p \to q$.

  For the converse, suppose that $\mb{X}$ does not satisfy~\eqref{ax:mp-corr}.
  Then we can find $x \in X$ and an admissible upset $a$ such that
  $x \in a$ but $x \notin s(x, a)$.
  Since in both full and descriptive conditional frames we have
  $s(x, a) = \bigcap \{ b \in A \mid s(x, a) \subseteq b \}$, we can
  find an admissible upset $b$ containing $s(x, a)$ but not $x$.
  The valuation $V$ with $V(p) = a$ and $V(q) = c$ yields a model
  falsifying $\axMP$.
\end{proof}

\begin{lemma}\label{lem:cs-corr}
  Let $\mb{X} = (X, \leq, s, A)$ be a full or descriptive general conditional frame.
  Then $\mb{X}$ validates $\axCS$ $(p \wedge q) \to (p \cto q)$
  if and only if for all $x \in X$ and $a \in A$:
  \begin{enumerate}\itemindent=3em
    \renewcommand{\labelenumi}{\textup{(}\theenumi\textup{)} }
    \renewcommand{\theenumi}{$\mathsf{cs}$-corr}
    \item \label{ax:cs-corr} $x \in a$ implies $s(x, a) \subseteq {\uparrow}x$
  \end{enumerate}
\end{lemma}
\begin{proof}
  Suppose $\mb{X}$ satisfies~\eqref{ax:cs-corr}.
  Let $x$ be a world and $V$ any valuation.
  If $x \models p \wedge q$ then $x \in V(p) \cap V(q)$,
  so $x \in V(p)$ and $x \in V(q)$. This implies ${\uparrow}x \subseteq V(q)$.
  By assumption $x \in V(p)$ implies $s(x, V(p)) \subseteq {\uparrow}x \subseteq V(q)$,
  so $x \models p \cto q$.

  For the converse, suppose that $\mb{X}$ does not satisfy~\eqref{ax:cs-corr}.
  Then there exist a world $x$ and an admissible upset $a$ such that
  $x \in a$ and $s(x, a) \not\subseteq {\uparrow}x$,
  so there is some $y \in s(x, a)$ with $x \not\leq y$.
  Since in a full or descriptive conditional frame we have
  ${\uparrow}x = \bigcap \{ b \in A \mid x \in b \}$,
  we can find an admissible ubset $b$ containing ${\uparrow}x$ but not $y$.
  Let $V$ be a valuation with $V(p) = a$ and $V(q) = b$.
  Then $x \models p \wedge q$ because $x \in a \cap b = V(p) \cap V(q)$,
  but $x \not\models p \cto q$ because $R_{V(p)}[x] = R_a[x] \not\subseteq b = V(q)$.
\end{proof}

\end{document}